\documentclass[12pt,a4paper,oneside]{amsart}
\usepackage{color,amssymb,latexsym,amsfonts,textcomp,fancyhdr,calc,graphicx}
\usepackage{pdfpages,amsfonts,amsthm,amssymb,latexsym,ucs,graphicx,leftidx,fancyhdr}
\usepackage{amsmath,amstext,amsthm,amssymb,amsxtra}
\usepackage[left=3cm,right=3cm,bottom=2cm,top=2cm]{geometry} 
\usepackage[colorlinks,citecolor=red,hypertexnames=false]{hyperref}
\usepackage{dsfont}
\usepackage[framemethod=tikz]{mdframed}
\usepackage[inline]{asymptote}
\usepackage{centernot}
\usepackage{enumitem}

\PassOptionsToPackage{linktocpage}{hyperref}

\usepackage[colorlinks,citecolor=red,hypertexnames=false]{hyperref}

\usepackage{txfonts,pxfonts,tikz} 
\usepackage{tgschola}
\usepackage[T1]{fontenc}

\newtheorem{theorem}{Theorem}
\newtheorem{corollary}[theorem]{Corollary}
\newtheorem{lemma}[theorem]{Lemma}

\newtheorem*{remark}{Remark}

\newtheorem{proposition}[theorem]{Proposition}
\newtheorem{definition}{Definition}[section]
\theoremstyle{definition}

\newcommand{\beql}[1]{\begin{equation}\label{#1}}
	\newcommand{\eeq}{\end{equation}}
\newcommand{\comment}[1]{}

\newcommand{\diff}{\textrm{d}}
\newcounter{rem}
\newcounter{step}
\newcounter{mysec}
\begin{document}
	\title[Large-time behavior of operators related to the fractional Laplacian]{Large-time behavior of two families of operators related to the fractional Laplacian on certain Riemannian manifolds}

	\begin{abstract}
		This note is concerned with two families of operators related to the fractional Laplacian, the first arising from the Caffarelli-Silvestre extension problem and the second from the fractional heat equation. They both include the Poisson semigroup. We show that on a complete, connected, and non-compact Riemannian manifold of non-negative Ricci curvature, in both cases, the solution with $L^1$ initial data behaves asymptotically as the mass times the fundamental solution. Similar long-time convergence results remain valid on
		more general manifolds satisfying the Li-Yau two-sided estimate of the heat kernel. The situation changes drastically  on hyperbolic space, and more generally on rank one non-compact symmetric spaces: we show that for the Poisson semigroup, the convergence to the Poisson kernel fails -but remains true under the additional assumption of radial initial data. 
	\end{abstract}
	
	
	\makeatletter
	\@namedef{subjclassname@2020}{\textnormal{2020}
		\it{Mathematics Subject Classification}}
	\makeatother

	\author{Effie Papageorgiou}
	\address{	Institut f{\"u}r Mathematik, Universit\"at Paderborn, Warburger Str. 100, D-33098
		Paderborn, Germany}
	\email{papageoeffie@gmail.com}
	
	\maketitle
	
	\tableofcontents
	
	\setlength{\parskip}{0.5em}
	
	\section{Introduction}
	Let $\mathcal{M}$ be a complete, non-compact Riemannian manifold and $\Delta$ be its Laplace-Beltrami operator. It is well understood that the long time behavior of solutions to the heat equation
	\begin{align}\label{S1 heat}
		\begin{cases}
			\partial_{t}u(t,x)&\,
			=\,\Delta u(t,x),
			\qquad\,t>0,\,\,x\in\mathcal{M},\\[5pt]
			u(0,x)&\,=\,f(x),
		\end{cases}
	\end{align} is strongly related to the global geometry of $\mathcal{M}$. This applies also to the heat kernel  $h_{t}\left( x,y\right) $, that is, the minimal
	positive fundamental solution of the heat equation or, equivalently, the
	integral kernel of the heat semigroup $\exp \left( t\Delta \right) $ (see
	for instance \cite{Gri2009}).
	
	The connection between the long time behavior of the
	solution $u (t, x)$ of (\ref{S1 heat}) for initial data $f\in L^1(\mathcal{M})$ with respect to 
	the Riemannian measure $\mu$ on $\mathcal{M}$ and that of the heat kernel $h_t(x, y)$  has recently been the subject of extensive studies, see for example \cite{APZ2023, GPZ2022, Vaz2019} or see \cite{AbAl2022, AGMP2021, APZ2023, P2023} for other Laplacians or settings.  Denote by $M=\int_{\mathcal{M}}f(x)\,\diff{\mu(x)}$ the mass of the initial data. In the case when $\mathcal{M}=\mathbb{R}^{n}$
	with the euclidean metric, the heat kernel is given by 
	\begin{equation*}
		h_{t}(x,y)\,=\,(4\pi {t})^{-\frac{n}{2}}e^{-\frac{|x-y|^{2}}{4t}}
	\end{equation*}%
	and the solution to (\ref{S1 heat}) satisfies as $t\rightarrow \infty $ 
	\begin{equation}
		\Vert u(t,\,.\,)\,-\,M\,h_{t}(\,.\,,x_{0})\Vert _{L^{1}(\mathbb{R}%
			^{n})}\,\longrightarrow \,0  \label{S1 L1 R}
	\end{equation}%
	and 
	\begin{equation}
		t^{\frac{n}{2}}\,\Vert u(t,\,.\,)\,-\,M\,h_{t}(\,.\,,x_{0})\Vert _{L^{\infty }(%
			\mathbb{R}^{n})}\,\longrightarrow \,0.  \label{S1 Linf R}
	\end{equation}%
	By interpolation, a similar convergence holds with respect to any $L^{p}$
	norm when $1<{p}<\infty $: 
	\begin{equation*}
		t^{\frac{n}{2p^{\prime }}}\,\Vert u(t,\,.\,)\,-\,M\,h_{t}(\,.\,,x_0)\Vert
		_{L^{p}(\mathbb{R}^{n})}\,\longrightarrow \,0
	\end{equation*}%
	where $p^{\prime }$ is the H\"{o}lder conjugate of $p$.
	
	Note that (\ref{S1 L1 R})  holds for \emph{any} choice
	of $x_{0}$, which means that in the long run the solution $u\left( t,x\right) $
	and the heat kernel $h_{t}\left( x,x_{0}\right) $ \textquotedblleft
	forget\textquotedblright\ about the initial function $f$, resp. initial
	point $x_{0}.$ We refer to a recent survey \cite{Vaz2017} for more details
	about this property in the euclidean setting. It is worth mentioning that this long-time asymptotic convergence result corresponds
	to the Central Limit Theorem of probability in the PDE setting.
	
	On manifolds of non-negative Ricci curvature, the results \eqref{S1 L1 R} and \eqref{S1 Linf R} were generalized in \cite{GPZ2022}. The situation is drastically different in hyperbolic spaces. It was shown by V\'{a}zquez \cite{Vaz2019} that (\ref{S1 L1 R}) fails for general absolutely integrable initial
	data $f$ but is still true if $%
	f$ is spherically symmetric around $x_{0}.$ Similar results were
	obtained in \cite{APZ2023} in a more general setting of symmetric spaces of
	non-compact type by using tools of harmonic analysis. Note that these spaces
	have nonpositive sectional curvature. Recall that in hyperbolic spaces
	Brownian motion $X_{t}$ tends to escape to $\infty $ along geodesics, which
	means that it \textquotedblleft remembers\textquotedblright\ at least the
	direction of the starting point $x_{0}.$ In \cite{GPZ2022}, it was also shown that (\ref{S1 Linf R}) fails on connected sums $\mathbb{R}^n\#\mathbb{R}^n$, $n\geq 3$.

	The \textit{fractional} Laplacian is the operator $(-\Delta)^{\sigma}$, $\sigma \in(0,1)$, defined as the spectral $\sigma$-th power of the Laplace-Beltrami operator, with $\text{Dom}(-\Delta)\subset\text{Dom}((-\Delta)^{\sigma})$. 
	It is connected to \textit{anomalous} diffusion, which accounts for much of the interest
	in modeling with fractional equations (quasi-geostrophic flows, turbulence and water waves, molecular dynamics,
	and relativistic quantum mechanics of stars). It also has various applications in probability and finance. On certain ``good'' non-compact Riemannian manifolds $\mathcal{M}$ (e.g. Cartan-Hadamard manifolds or manifolds with non-negative Ricci curvature,  see \cite[Proposition 3.3]{BanEtAl}) one can obtain the fractional Laplacian through a Dirichlet-to-Neumann map extension problem
	introduced by Caffarelli and Silvestre \cite{CaSi2007}, as well as a Poisson formula and a fundamental solution, see the work of Stinga and Torrea  \cite{ST2010}. 
	More precisely, let $H^{\sigma}(\mathcal{M})$ denote the usual Sobolev space on $\mathcal{M}$. Then for any given $f\in H^{\sigma}(\mathcal{M})$ there exists a unique solution of the extension problem
	\begin{equation}\label{CS}
		\Delta v+\frac{(1-2\sigma)}{t}\frac{\partial v}{\partial t}+\frac{\partial^2 v}{\partial t^2}=0, \quad 0 <\sigma < 1, \quad t>0, \; x\in \mathcal{M},
	\end{equation}
	with $v(0,x)=f(x)$ and the fractional Laplacian can be recovered through
	\begin{equation*}
		(-\Delta)^{\sigma}f(x)=-2^{2\sigma-1}\frac{\Gamma(\sigma)}{\Gamma(1-\sigma)}\lim_{t\rightarrow 0^{+}}t^{1-2\sigma}\,\frac{\partial v}{\partial t}(x,t).
	\end{equation*}
	Notice that equation \eqref{CS} gives rise to the first family of operators this note is concerned about. The second family of operators we consider arises from the fractional heat equation
	\begin{equation}\label{FH}
		\partial_t u + (-\Delta )^{\alpha/2}u = 0, \quad 0 <\alpha < 2, \quad t>0, \; x\in \mathcal{M}.
	\end{equation} 
	These two families of operators have drawn much attention, see for instance \cite{AGMP2021, BanEtAl, BP2022, BG1960, BJ07, CK03, CK08, Shi17, Vaz2018} and the references therein. It is worth mentioning that both families of operators include the Poisson semigroup (for  $\sigma=1/2$ and $\alpha=1$  respectively).

	The aim of this paper is to study the long time behavior of these two families of operators on certain Riemannian manifolds, for absolutely integrable initial data. More precisely, we treat the case of  non-negative Ricci curvature and generalizations of these, that is, doubling volume, complete, non-compact manifolds with double-sided heat kernel estimates of the Li-Yau type, and show that the results
	are essentially euclidean, in the sense of convergence proved in \cite{Vaz2018}. This is no longer the case in negatively curved
	manifolds. More precisely, we consider the Poisson semigroup on real hyperbolic
	space, and more generally on rank one symmetric spaces of non-compact
	type and show that in this case, the long time results are vastly different:
	the aforementioned euclidean-type results fail even for  compactly supported
	initial data.  
	
	Notice that for all manifolds considered here, these two families of operators admit integral kernels which are actually probability measures.
	
	Our main result is the following.

	\begin{theorem}\label{MainThm} 
		Let $\mathcal{M}$ be a complete, connected and
		non-compact Riemannian manifold of non-negative Ricci curvature. Let $\psi_t^{\gamma}$, $\gamma\in\{1,\alpha\}$, be either the fundamental solution to (\ref{CS}) for $\gamma=1$ for all $\sigma\in (0,1)$ or the fundamental solution to (\ref{FH}) for $\gamma=\alpha$,  $\alpha\in(0,2)$, and consider the corresponding integral operators
		$$\mathcal{K}_t^{\gamma}(f)(x)=\int_{\mathcal{M}}\psi_t^{\gamma}(x,y)\,f(y)\, \diff \mu(y), \quad f\in L^1(\mathcal{M}).$$
		Set $M=\int_{\mathcal{M}}f\, d\mu$ and fix a base
		point $x_{0}\in \mathcal{M}$.  Then,  as $%
		t\rightarrow +\infty $,
		\begin{equation}
			\Vert \mathcal{K}_t^{\gamma}(f)-\,M\,\psi_t^{\gamma}(\,.\,,x_{0})\Vert _{L^{1}(\mathcal{M}%
				)}\,\longrightarrow \,0  \label{CSL1}
		\end{equation}%
		and
		\begin{equation}
			\,\Vert \left\vert \mathcal{K}_t^{\gamma}(f)-\,M\,\psi_t^{\gamma}(\,.\,,x_{0})\right\vert V(\,.\,,
			t^{1/\gamma})\Vert _{L^{\infty }(\mathcal{M})}\,\longrightarrow 0.  
			\label{CSsup}
		\end{equation}
	\end{theorem}
	
	\begin{remark}
		By interpolation between (\ref{CSL1}) and (\ref{CSsup}%
		), we obtain for any $p\in \left( 1,\infty \right) $ 
		\begin{equation*}
			\,\Vert \left\vert \mathcal{K}_t^{\gamma}(f)-\,M\,\psi_t^{\gamma}(\,.\,,x_{0})\right\vert V(\,.\,,%
			t^{1/\gamma})^{1/p^{\prime }}\Vert _{L^{p}(\mathcal{M})}\,\longrightarrow \,0%
			\mathnormal{.}  
		\end{equation*}
	\end{remark}
	
	The situation changes drastically in real hyperbolic space, and generally, in rank one non-compact symmetric spaces. More precisely, the Poisson semigroup fails to satisfy these convergences.

	\begin{theorem}\label{Main Poisson}
		Let $\mathbb{X}$ be a rank one non-compact symmetric space.	Assume that $f\in L^1(\mathbb{X})$ and let $M=\int_{\mathbb{X}}f$ denote its mass. If $e^{-t\sqrt{-\Delta}}$ is the Poisson semigroup and $p_t$ the Poisson kernel, then, in general 
		$$
		\|  e^{-t\sqrt{-\Delta}}f - M \, p_t \|_{L^1(\mathbb{X})}\centernot\longrightarrow 0
		\quad\text{as}\quad
		t \to +\infty .
		$$
		However, the convergence holds if $f$ is  radial. 
	\end{theorem}
	
	To the best of our knowledge, even though the fractional Laplacian and related diffusion equations have been the center of many studies, the above asymptotics have been established only on euclidean space, \cite{Vaz2018}. We use different methods to prove our results. For the case of manifolds with non-negative Ricci curvature, we use subordination formulas to use information for the heat kernel, such as double-sided bounds and a quantitative H{\"o}lder continuity estimate, which is a key component to our proof. Notice that one could have pursued gradient estimates, but the approach used here is more general. For the case of rank one non-compact symmetric spaces, we rely on tools of harmonic analysis available in this setting and large-time asymptotics of the Poisson kernel. An essential idea of the proof in the rank one case is to describe the critical region of the Poisson kernel, since this kernel is a probability measure.

	\section{Preliminaries} \label{Prelim}
	From now on, $\mathcal{M}$ denotes a complete,
	connected, non-compact Riemannian manifold of dimension $n\geq 2$. Let $\mu $
	be the Riemannian measure on $\mathcal{M}$. Let $d(x,y)$ be the geodesic
	distance between two points $x,y\in \mathcal{M}$, and $V(x,r)=\mu \left(
	B\left( x,r\right) \right) $ be the Riemannian volume of the geodesic ball $%
	B(x,r)$ of radius $r$ centered at $x\in \mathcal{M}$.
	
	Throughout the paper we follow the convention that $C,C_{1},c,c_{1}...$ denote
	positive constants.
	These constants may depend on $\mathcal{M}$ but do not depend on the
	variables $x,y,t$. Moreover, the notation $A\lesssim {B}$ between two
	positive expressions means that $A\leq {C}B$, and $A\asymp {B}$ means $%
	cB\leq {A}\leq {C}B$. Also, $A(t)\sim B(t)$ means that $A(t)/B(t)\rightarrow 1$ as $t\rightarrow +\infty$.

	We say that $\mathcal{M}$ satisfies the volume doubling property if, for all 
	$x\in \mathcal{M}$ and $r>0$, we have 
	\begin{equation}
		V(x,2r)\,\leq \,C\,V(x,r).  \label{VD}
	\end{equation}%
	It follows from (\ref{VD}) that there exist some positive constants $\nu
	,\nu ^{\prime }>0$ such that 
	\begin{equation}
		c\,\left( \frac{R}{r}\right) ^{\nu ^{\prime }}\,\leq \,\frac{V(x,R)}{V(x,r)}%
		\,\leq \,C\,\left( \frac{R}{r}\right) ^{\nu }  \label{S2 Comp same center}
	\end{equation}%
	for all $x\in \mathcal{M}$ and $0<r\leq {R}$ (see for instance \cite[Section
	15.6]{Gri2009}). Moreover, (\ref{S2 Comp same center}) implies that, for all 
	$x,y\in \mathcal{M}$ and $r>0,$ 
	\begin{equation*}
		\frac{V\left( x,r\right) }{V\left( y,r\right) }\leq C\left( 1+\frac{d\left(
			x,y\right) }{r}\right) ^{\nu }.  \label{S2 Comp diff center}
	\end{equation*}
	Notice that hyperbolic spaces as well as all non-compact symmetric spaces fail to be doubling (they are locally doubling, though). More precisely, in the rank one case, if $n=\dim \mathbb{X}$ and $\rho^2>0$ is the bottom of the spectrum, it holds 
	\begin{align}\label{volX}
		V(x,r)\asymp\begin{cases}
			r^n,  &\text{if} \;0<r<1\, ,\\[5pt]
			e^{2\rho r}, &\text{if} \;r\geq 1.
		\end{cases}
	\end{align}
	
	The integral kernel $h_{t}\left( x,y\right) $ of the heat semigroup $\exp
	(t\Delta )$ is the smallest positive fundamental solution to the heat
	equation (\ref{S1 heat}). It is known that $h_{t}\left( x,y\right) $ is
	smooth in $\left( t,x,y\right) $, symmetric in $x,y$, and
	satisfies the semigroup identity (see for instance  \cite{Gri2009}, \cite{Str1983}). Besides, for all $y\in\mathcal{M}$ and $t>0$ 
	\begin{equation*}
		\int_{\mathcal{M}}h_{t}\left( x,y\right) \mathrm{d}\mu (x)\leq 1.
	\end{equation*}%
	The manifold $\mathcal{M}$ is called stochastically complete if for all $%
	y\in\mathcal{M}$ and $t>0$ 
	\begin{equation*}
		\int_{\mathcal{M}}h_{t}(x,y)\,\mathrm{d}\mu (x)=\,1.
	\end{equation*}%
	It is known that if $\mathcal{M}$ is geodesically complete and, for some $%
	x_{0}\in \mathcal{M}$ and all large enough $r,$%
	\begin{equation*}
		V\left( x_{0},r\right) \leq e^{Cr^{2}},
	\end{equation*}%
	then $\mathcal{M}$ is stochastically complete. In particular, the volume
	doubling property \eqref{S2 Comp same center} and the  volume bounds \eqref{volX} for rank one symmetric spaces imply that all manifolds considered in this paper are stochastically complete.

	When the Ricci curvature of $\mathcal{M}$ is non-negative, 
	the following two-sided estimates of the heat kernel were proved by Li and
	Yau \cite{LiYa1986}:
	\begin{equation}
		\frac{c_{1}}{V(y,\sqrt{t})}\,\exp \Big(-C_{1}\,\frac{d^{2}(x,y)}{t}\Big)%
		\,\leq \,h_{t}(x,y)\,\leq \,\frac{C_{2}}{V(y,\sqrt{t})}\,\exp \Big(-c_{2}\,%
		\frac{d^{2}(x,y)}{t}\Big). \label{LY}
	\end{equation}%
	Apart from manifolds with non-negative Ricci curvature, the
	above-described manifolds cover many other examples.
	Let us recall that, on a complete Riemannian manifold,
	the following three properties are equivalent:
	
	\begin{itemize}
		\item The two-sided estimate \eqref{LY} of the heat kernel;
		
		\item
		The uniform parabolic Harnack inequality:
		\begin{align}\label{PHI}
			\sup_{(\frac{T}{4},\frac{T}{2})\times{B(x,\frac{r}{2})}}u(t,x)\;
			\le\,C\,
			\sup_{(\frac{3T}{4},T)\times{B(x,\frac{r}{2})}}u(t,x),
		\end{align}
		where $u(t,x)$ is a non-negative solution of the heat equation 
		$\partial _{t}u=\Delta u$ in a cylinder $(0,T)\times{B(x,r)}$ 
		with $x\in\mathcal{M}$, $r>0$ and $T=r^{2}$.
		
		\item 
		The conjunction of the volume doubling property \eqref{VD} and the
		Poincaré inequality:
		\begin{align}	\label{PI}
			\int_{B(x,r)}|f-f_{B}|^{2}\,\diff{\mu}\,
			\le\,C\,r^{2}\,
			\int_{B(x,r)}|\nabla{f}|^{2}\,\diff{\mu},
		\end{align}
		for all $x\in\mathcal{M}$, $t>0$, and bounded Lipschitz functions $f$ in
		$B(x,r)$. Here, $f_{B}$ is the mean of $f$ over $B(x,r)$.
	\end{itemize}
	See, for instance, \cite {FaSt1986,Gri1991,Sal1992,Sal2002} for more details.
	Manifolds satisfying these equivalent conditions include complete manifolds 
	with non-negative Ricci curvature, connected Lie groups with polynomial 
	volume growth, co-compact covering manifolds whose deck transformation 
	has polynomial growth, and many others. 
	We refer to \cite[pp.417--418]{Sal2010} for a list of examples.

	\begin{corollary}\label{CorGen}
		Let $\mathcal{M}$ be a geodesically complete non-compact
		manifold that satisfies one of the following equivalent conditions:
		\begin{itemize}
			\item the two-sided estimate \eqref{LY} of the heat kernel;
			
			\item the uniform parabolic Harnack inequality \eqref{PHI};
			
			\item the conjunction of the volume doubling property \eqref{VD} and the
			Poincar{\'e} inequality \eqref{PI}.
		\end{itemize}
		
		Then the conclusions of Theorem \ref{MainThm} are true.
	\end{corollary}

	Let us now recall a consequence of the two-sided estimate \eqref{LY}, which will be essential for this note, see for instance  \cite[Theorem 5.4.12]{Sal2002}:
	there exists $0<\theta\leq1$ such that for all $t>0$,
	$x,y,z\in\mathcal{M}$, and $d(y,z)\le\sqrt{t}$,
	\begin{align}
		|h_{t}(x,y)\,-\,h_{t}(x,z)|\,
		\le\,
		\Big(\frac{d(y,z)}{\sqrt{t}}\Big)^{\theta}\,
		\frac{C}{V(x,\sqrt{t})}\,
		\exp\Big(-c\frac{d^2(x,y)}{t}\Big).
		\label{Holder}
	\end{align}
	Notice that a pointwise gradient estimate of the form
	\begin{equation*}
		|\nabla _{y}\,h_{t}(x,y)|\,\leq \,\frac{C}{\sqrt{t}\,V(x,\sqrt{t})}\,\exp %
		\Big(-c\,\frac{d^{2}(x,y)}{t}\Big)
	\end{equation*} 
	is the limit case $\theta=1$ of \eqref{Holder}, but requires more structure on the manifold, such as non-negative Ricci curvature, see for instance \cite{LiYa1986}. In this sense, the H{\"o}lder continuity estimate \eqref{Holder} is more general.

	Last, recall that on spaces of essentially negative curvature the
	above estimates of the heat kernel typically fail: for example, these are
	hyperbolic spaces \cite{DaMa1988}, non-compact symmetric spaces \cite{AJ99, AO}, 
	asymptotically hyperbolic manifolds \cite{ChHa2020}, and fractal-like manifolds 
	\cite{Bar1998}. 
	
	We finally recall the definition of real hyperbolic space and more generally, of rank one symmetric spaces, as well as some indispensable tools from Fourier analysis on these spaces.

	\subsection{Rank one non-compact symmetric spaces} \label{subsection SS}
	Our main reference for this subsection is \cite{Hel}.
	Let $\mathbb{G}$ be a connected,  non-compact semisimple Lie group with finite center. Let $K$ be a maximal compact subgroup of $\mathbb{G}$ and $\mathbb{X}=\mathbb{G}/K$ be the corresponding symmetric space. We consider a Cartan decomposition $\mathfrak{g}=\mathfrak{k}\oplus\mathfrak{p}$ of the Lie algebra of $\mathbb{G}$. Fix a maximal abelian subspace $\mathfrak{a}$ of $\mathfrak{p}$ and consider the decomposition $\mathfrak{g}=\mathfrak{n}\oplus\mathfrak{a}\oplus\mathfrak{k}$. If $\mathfrak{a}\cong \mathbb{R}$, then we say that the symmetric space $\mathbb{X}$ has rank one. 
	
	From now on, we assume that $\mathrm{rank}\mathbb{X}=1$. In this case, after fixing some order on the non-zero restricted roots, there are at most two roots which
	are positive with respect to this order, which we denote by $\alpha$ and $2\alpha$. Let $m_{\alpha}$ and $m_{2\alpha}$ be the
	multiplicities of these roots, and define the number $\rho$ by  $\rho:=(m_{\alpha}+2m_{2\alpha})/2$. A rank one non-compact symmetric space
	is one of the following: the real, the complex, the quaternionic hyperbolic space and the octonionic hyperbolic plane.  We have $n=\dim \mathbb{X}=m_{\alpha}+m_{2\alpha}+1$ and $\rho$ equal to $(n-1)/2$, $n/2$, $n/2+1$ and $11$, respectively.
	
	The group $\mathbb{G}$ admits the following decompositions,
	\begin{align*}
		\begin{cases}
			\,\mathbb{G}\,=\,N\,\mathbb{A}\,K 
			\qquad&\textnormal{(Iwasawa)}, \\[5pt]
			\,\mathbb{G}\,=\,K\,\overline{\mathbb{A}^{+}}\,K
			\qquad&\textnormal{(Cartan)}.
		\end{cases}
	\end{align*}
	Let $H_0$ be the unique element of $\mathfrak{a}$ with the property that $\langle \alpha, H_0 \rangle=1$ and normalize
	the Killing form on $\mathfrak{g}$ such that $|H_0| = 1$. Denote by $\tau(g)$ the real number such that
	$$g=n\exp (\tau(g) \, H_0)k.$$ To simplify the notation,
	we often identify the Lie subgroup $\mathbb{A}=\exp \mathfrak{a}$ with the real line $\mathbb{R}$ using the map $\tau \mapsto \exp(\tau\,H_0)$. Notice that we may also identify  $\overline{\mathbb{A}^{+}}$ with $[0, +\infty)$. In the Cartan decomposition, the Haar measure 
	on $\mathbb{G}$ writes
	\begin{align}\label{vol}
		\int_{\mathbb{G}}f(g)\,\diff{g}
		=\,
		\textrm{const.}\,\int_{K}\diff{k_1}\,
		\int_{0}^{\infty}\delta(r)\, \diff{r}
		\int_{K}f(k_{1}(\exp r\, H_0)k_{2})\,\diff{k_2},
	\end{align}
	with density
	\begin{equation}\label{dens}
		\delta(r) =(\sinh r)^{m_{\alpha}}(\sinh 2r)^{m_{2\alpha}} \lesssim e^{2\rho \,r}.
	\end{equation}
	Here $K$ is equipped with its normalized Haar measure and ``const'' is a positive normalizing constant, so that for right-$K$ invariant functions, we have
	$$\int_{\mathbb{X}}f(x)\,\diff\mu(x)=\int_{\mathbb{G}}f(g)\,\diff g.$$
	Notice that viewed on $\mathbb{G}/K$, $r$ is the distance of $gK$ to the origin $o=\{K\}$. 
	
	Finally, we describe Fourier analysis on rank one non-compact symmetric spaces.  For continuous compactly supported functions, the Helgason-Fourier transform is defined by
	\begin{align}\label{H-Ftr}
		\widehat{f}(\lambda,k\mathbb{M})=\,\int_{\mathbb{G}}\,f(gK)\,
		e^{(-i\lambda+\rho)\,\tau(k^{-1}g)}\,\diff{g}, \quad \lambda \in \mathbb{C}, \; k\in K.
	\end{align}
	Here, $\mathbb{M}$ denotes the centralizer of $\exp\mathfrak{a}$ in $K$. Let us also define the spherical transform of continuous compactly supported and radial functions by
	\begin{align*}
		\mathcal{H}f(\lambda)=\!\int_{\mathbb{G}}f(gK)\,\varphi_{-\lambda}(g)\, \diff{g},
	\end{align*}
	where $\varphi_{\lambda}$ is the
	elementary spherical function of index $\lambda\in\mathbb{C}$. The functions $\varphi_{\lambda}$ are normalized eigenfunctions of $\Delta$, that is, $\Delta \varphi_{\lambda}=-(\lambda^2+\rho^2)\varphi_{\lambda}$ with $\varphi_{\lambda}(o)=1$. They are also radial and have the property that $\varphi_{\lambda}=\varphi_{-\lambda}$. Finally, let us recall that in the case of radial Schwartz functions, we have $$\widehat{f}(\lambda,k\mathbb{M})=\mathcal{H}f(\lambda).$$

	\section{Fractional Laplacian and semigroups subordinated to the heat semigroup} \label{section3}
	This section deals with two families of operators related to the fractional Laplacian but both subordinated via integral representations to the heat semigroup. Interestingly, the Poisson semigroup belongs to both, for special values of their parameters. 
	
	In recent years there has been intensive research on various kinds of fractional order operators.
	Being nonlocal objects, local PDE techniques to treat nonlinear problems for the fractional operators do not apply. To overcome this difficulty, in the euclidean case, Caffarelli and Silvestre
	\cite{CaSi2007} studied the extension problem associated with the Laplacian and realized the fractional
	power as the map taking Dirichlet data to Neumann data. In \cite{ST2010} Stinga and Torrea related
	the extension problem for the fractional Laplacian to the heat semigroup, providing a subordination formula and conditions for the existence of an integral kernel. On certain classes
	of non-compact manifolds, which include symmetric spaces of non-compact type, the extension problem has been studied by Banica, González and Sáez \cite{BanEtAl}. Interestingly, in the non-compact
	setting one needs to have a precise control of the behavior of the metric at infinity and
	geometry plays a crucial role.
	
	To begin with, using the spectral theorem, one can define fractional powers of the Laplacian via the heat semigroup,
	$$(-\Delta)^{\sigma}f(x)=\int_{0}^{\infty}(e^{u\Delta}f(x)-f(x))\frac{du}{u^{1+\sigma}} \quad \text{ in } L^2(\mathcal{M}), \; f\in \text{Dom}(-\Delta). $$
	see \cite[(5), p.260]{Y}. Then, the relation between the fractional Laplacian and the extension problem (\ref{CS}) is the following.
	\begin{theorem} \cite{ST2010}. Let $\sigma\in (0,1)$. Then for  $f\in \text{Dom}((-\Delta)^{\sigma})$, a solution to the extension problem
		\begin{align}\label{extension}
			\Delta v+\frac{(1-2\sigma)}{t}\frac{\partial v}{\partial t}+\frac{\partial^2 v}{\partial t^2}=0, \quad  
			v(0,x)\,=\,f(x),\quad  t>0,\,\; x\in \mathcal{M},
		\end{align}
		is given by
		\begin{equation}\label{ST 1.9}
			T^{\sigma}_tf(x):=v(t,x)=\frac{t^{2\sigma}}{2^{2\sigma}\Gamma(\sigma)}\int_{0}^{+\infty}e^{u\Delta}f(x)\,e^{-\frac{t^2}{4u}}\frac{du}{u^{1+\sigma}}.
		\end{equation}
		Moreover, the fractional Laplacian  on $\mathcal{M}$ can be recovered through
		\begin{equation*}(-\Delta)^{\sigma}f(x)=-2^{2\sigma-1}\frac{\Gamma(\sigma)}{\Gamma(1-\sigma)}\lim_{t\rightarrow 0^{+}}t^{1-2\sigma}\frac{\partial v}{\partial t}(x,t).
		\end{equation*}
	\end{theorem}
	From a probabilistic point of view, the extension problem corresponds to the property
	that all symmetric stable processes can be obtained as traces of degenerate Bessel
	diffusion processes, see \cite{Stinga}. 
	
	However, despite the subordination of $\{T_t^{\sigma}\}_{t>0}$ to the heat semigroup, passing from the heat kernel to a Poisson kernel  is a non-trivial
	issue in the case of non-compact manifolds since one needs to control the behavior at
	infinity.  By \cite{BanEtAl, ST2010} and under the description of the heat semigroup given in the present paper, one needs to check 
	whether, given $x_0$, there exists a constant $C_{x_0}$ and $ \epsilon > 0$ such that the heat kernel on the manifold $\mathcal{M}$ satisfies
	\begin{equation}\label{STheat}
		\|h_t(\,.\,x_0)\|_{L^2(\mathcal{M})}+\|\partial_th_t(\,.\,x_0)\|_{L^2(\mathcal{M})} \leq C_{x_0}(1+t^{\epsilon})t^{-\epsilon}.
	\end{equation}
	Thus the problem of an integral kernel for the operator $T_t^{\sigma}$ reduces to obtaining suitable upper bounds for the heat
	kernel (from where one may derive information for its time derivatives as well, see for instance \cite{Gri1995}). Inequality \eqref{STheat} is true on real hyperbolic space as well as on non-compact symmetric spaces of arbitrary rank. It is also true on manifolds satisfying a volume doubling condition and the local Poincar{\'e} inequality \cite[Proposition 3.3]{BanEtAl}, such as manifolds of non-negative Ricci curvature, or more generally the manifolds considered in Corollary \ref{CorGen}. Then, the function $T_t^{\sigma}f$ in \eqref{ST 1.9} is given by
	$$T_t^{\sigma}f(x)=v(t,x)=\int_{\mathcal{M}} Q_t^{\sigma}(x,y)f(y)\,\diff{\mu(y)}, $$
	where the integral kernels are given by
	\begin{equation}\label{kernelCS}
		Q_t^{\sigma}(x,y)=\frac{t^{2\sigma}}{2^{2\sigma}\Gamma(\sigma)}\int_{0}^{+\infty}h_u(x,y)\,e^{-\frac{t^2}{4u}}\frac{\diff{u}}{u^{1+\sigma}}.
	\end{equation}

	We now pass to the fractional heat equation,  again for the manifolds mentioned above. Let $\alpha \in (0,2)$ and take $\eta_{t}^{\alpha}$ to be the inverse Laplace transform of the function $\exp\{-t\,(.)^{\alpha/2}\}$.
	The fractional Laplacian $(-\Delta)^{\alpha/2}$ is the infinitesimal generator of a standard isotropic $\alpha$-stable L{\'e}vy motion $X^{\alpha}_t$.
	This process is a L{\'e}vy process, which can be viewed as the long-time scaling limit of a random walk with power law jumps (\cite[Theorem 6.17]{Mesi2011}). Via subordination to the heat semigroup, we may write 
	$$e^{-t(-\Delta)^{\alpha/2}}=\int_0^{\infty}e^{u\Delta}\,\eta_{t}^{\alpha}(u)\,\diff{u},$$ 
	\cite[(7), p.260]{Y}. Then, for the manifolds considered in this note (those of Corollary \ref{CorGen} and rank one non-compact symmetric spaces), for any reasonable $f$ we have that 
	$$W_t^{\alpha}f(x):=e^{-t(-\Delta)^{\alpha/2}}f(x)=w(t,x)=\int_{\mathcal{M}} P_t^{\alpha}(x,y)f(y)\,\diff{\mu(y)}, $$
	where the kernels are given by
	\begin{equation}\label{kernelFH}
		P_t^{\alpha}(x,y)=\int_0^{\infty}h_u(x)\,\eta_{t}^{\alpha}(u)\,\diff{u}.
	\end{equation}
	The  family $\{W^{\alpha}_t\}_{t>0}$ is a $C_0$- semigroup, \cite{Y}. For every $L^1(\mathcal{M})$ (for an optimal class of initial data on euclidean space, see \cite{BSV17}), the function $w(t,\, . \,)=W^{\alpha}_tf$ solves the initial value problem 
	\begin{align}\label{fractionalheat}
		\partial_{t}w(t,x)+(-\Delta)^{\alpha/2} w(t,x)=0, \quad 
		w(0,x)\,=\,f(x), \quad t>0,\,\,x\in\mathcal{M}.
	\end{align} 
	In the case of hyperbolic spaces, the $\alpha$-stable process  with transition densities
	$P_t^{\alpha}(x,y)$  was first defined by Getoor \cite{G61} (see also \cite{GS04} for general non-compact symmetric spaces). 
	
	Finally, observe that owing to the subordination formulas and the properties of the heat kernel, both kernels $Q_t^{\sigma}$, $P_t^{\alpha}$ are non-negative and  symmetric. They are also probability measures: to see this, recall first that the manifolds considered here are stochastically complete.  Then, by a Fubini argument, the claim follows for $Q_t^{\sigma}$  by the definition of Gamma function, while for $P_t^{\alpha}$ from the fact that $\int_{0}^{\infty}\eta_{t}^{\alpha}(u)\, \diff{u}=1$, \cite[Eq (14), p.262]{Y}.

	\section{Asymptotics for two operators related to the fractional Laplacian:
	\linebreak
		the case of non-negative Ricci curvature}
	Throughout this section, we consider $\mathcal{M}$ to be a manifold of non-negative Ricci curvature -or more generally, we consider the Riemannian manifolds of Corollary \ref{CorGen}- and study the long-time properties of the families of operators considered in Section \ref{section3}.  An essential idea of the proof is the use of the H{\"o}lder continuity of the heat kernel, owing to subordination formulas of their kernels. 
	
	\subsection{A suitable class of kernels}\label{sec:class}
	
	We begin by introducing a suitable class of kernels to unify our approach for the families of operators $\{T_t^{\sigma}\}_{t>0}$, $\sigma\in(0,1)$ and $\{W_t^{\alpha}\}_{t>0}$, $\alpha \in (0,2)$. 
	
	\begin{definition}
		Suppose $\gamma>0$. We say that a family of measurable functions $(\psi_t^{\gamma})_{t>0}$ on $\mathcal{M}$ \textit{belongs to the class $\mathcal{P}_{\gamma}$}, and write $(\psi_t^{\gamma}) \in \mathcal{P}_{\gamma}$,  if 
		\begin{enumerate}
			\item[(P1)]\label{P1} For all $t>0$, $\psi_t^{\gamma}$ is positive and symmetric on $\mathcal{M}$, i.e. $0<\psi_{t}^{\gamma}(x,y)=\psi_{t}^{\gamma}(y,x)$ for all $x,y\in \mathcal{M};$ 
			
			\item[(P2)]\label{P2} For all $t>0$, for all $x\in \mathcal{M}$,  it holds
			$$\int_{\mathcal{M}}\psi_t^{\gamma}(x,y)\,\diff{\mu}(y)=1  \quad \text{and} \quad  \|\psi_t^{\gamma}(x,\, .\,)\|_{L^{\infty}(\mathcal{M})}\asymp V(x,t^{1/\gamma})^{-1};$$
			
			\item[(P3)]\label{P3} 	There is a constant $C > 1$ such that for all $x,y,x_0\in \mathcal{M}$ such that $d(x_0,y) \leq t^{1/\gamma}$, we have
			$$C^{-1} \leq
			\frac{\psi_t^{\gamma}(x,y)}{\psi_t^{\gamma}(x,x_0)}\leq C;$$
			
			\item[(P4)]\label{P4}  There is a constant $\theta_{\gamma} > 0$ such that $d(x_0,y) \leq \xi$ implies there is $t_0(\xi, \gamma)>0$ such that
			$$|\psi_t^{\gamma}(x,x_0)-\psi_t^{\gamma}(x,y)|\leq C(\xi,\gamma,\mathcal{M})\,t^{-\theta_{\gamma}}\,\psi_t(x,x_0)$$
			for all $x\in \mathcal{M}$, for all $t>t_0(\xi, \gamma)$.
		\end{enumerate}
	\end{definition}
	
	\begin{theorem}\label{MainThmCS} 
		Fix $\gamma>0$. Let $(\psi_{t}^{\gamma})\in \mathcal{P}_{\gamma}$, and consider the integral operator
		$$\mathcal{K}_t^{\gamma} (f)(x)=\int_{\mathcal{M}}\psi_t^{\gamma}(x,y)\,f(y)\,\diff\mu(y), \quad x\in \mathcal{M}$$  acting on functions
		$f\in {L}^{1}(\mathcal{M})$. Set $M=\int_{\mathcal{M}}f\, d\mu$ and fix a basepoint $x_{0}\in \mathcal{M}$.  Then,  as $%
		t\rightarrow +\infty $,
		\begin{equation}
			\Vert \mathcal{K}_t^{\gamma}(f)-\,M\,\psi_t^{\gamma}(\,.\,,x_{0})\Vert _{L^{1}(\mathcal{M}%
				)}\,\longrightarrow \,0  \label{CSL1}
		\end{equation}%
		and
		\begin{equation}
			\,\Vert \left\vert \mathcal{K}_t^{\gamma}(f)\,-\,M\, \psi_t^{\gamma}(\,.\,,x_{0})\right\vert V(\,.\,,
			t^{1/\gamma})\Vert _{L^{\infty }(\mathcal{M})}\,\longrightarrow 0.  
			\label{CSsup}
		\end{equation}
	\end{theorem}
	
	\begin{remark}
		By convexity between (\ref{CSL1}) and (\ref{CSsup}%
		), we obtain for any $p\in \left( 1,\infty \right) $ 
		\begin{equation*}
			\,\Vert \left\vert \mathcal{K}_t^{\gamma}(f)(\,.\,)\,-\,M\,\psi_t^{\gamma}(\,.\,,x_{0})\right\vert V(\,.\,,%
			t^{1/\gamma})^{1/p^{\prime }}\Vert _{L^{p}(\mathcal{M})}\,\longrightarrow \,0%
			\mathnormal{.}  
		\end{equation*}
	\end{remark}

	Let us stress that the conditions (P1)-(P4) are not necessarily optimal. The class $\mathcal{P}_{\gamma}$ and the above conditions will simply allow us to avoid repeating several steps when proving convergence to the fundamental solution for the two families of operators $\{T_t^{\sigma}\}_{t>0}$, $\sigma\in(0,1)$ and $\{W_t^{\alpha}\}_{t>0}$, $\alpha \in (0,2)$. 
	
	To prove Theorem \ref{MainThmCS}, it is sufficient to consider the action of the operator $\mathcal{K}_t^{\gamma}$ on continuous
	compactly supported functions $f$. Then, owing to the properties (P1)-(P4), one can show that the desired convergence remains
	valid for the whole class of $L^{1}(\mathcal{M})$ initial data by using a density argument, see for instance the arguments in \cite[pp.17-18]{APZ2023} or \cite[pp.11-12]{GPZ2022}. Therefore, it suffices to prove the following result, which also gives a rate of convergence for continuous compactly supported initial data.
	\begin{proposition}
		\label{S3 proposition}
		Fix $\gamma>0$ and a basepoint $x_{0}\in \mathcal{M}$. Let $f\in \mathcal{C}%
		_{c}(B(x_{0},\xi))$ for some $\xi>0$ and set $M=\int_{\mathcal{M}}f\, d\mu$. Then, for all $t>t_0(\xi, \gamma)$, it holds
		\begin{equation*}
			\Vert \mathcal{K}_t^{\gamma}(f)-\,M\, \psi_{t}^{\gamma}(\,.\,,x_{0})\Vert _{L^{1}(\mathcal{M})}\,\leq C(\xi,\gamma,\mathcal{M})\,t^{-\theta_{\gamma}}
			\label{L1 conv Cc}
		\end{equation*}
		and 
		\begin{equation*}
			\,\Vert \left\vert \mathcal{K}_t^{\gamma}(f)\,-\,M\, \psi_t^{\gamma}(\,.\,,x_{0})\right\vert V(\,.\,,
			t^{1/\gamma})\Vert _{L^{\infty }(\mathcal{M})}\leq C(\xi,\gamma,\mathcal{M})\,t^{-\theta_{\gamma}}.
			\label{Linf conv Cc}
		\end{equation*}
	\end{proposition}
	
	\begin{proof}

		First of all, observe that the operator $\mathcal{K}^{\gamma}_{t}$ is bounded on $L^1(\mathcal{M})$, due to (P1) and (P2).
		Write
		\begin{align}
			\mathcal{K}_t^{\gamma}(f)(x)\,-\,M\,\psi_{t}^{\gamma}(x,x_{0})\,& =\,\int_{\mathcal{M}}\,f(y)\,\left(
			\psi_{t}^{\gamma}(x,y)\,-\psi_{t}^{\gamma}(x,x_{0})\right)\,\diff{\mu(y)}  \notag \\
			& =\int_{B\left( x_{0},\xi\right) }\,f(y)\,\left(
			\psi_{t}^{\gamma}(x,y)\,-\psi_{t}^{\gamma}(x,x_{0})\right)\,\diff{\mu(y)}.  \label{comment}
		\end{align}

		Therefore, by \eqref{comment}, (P1), (P4) and by integrating in $x$ over $\mathcal{M}$, we obtain 
		\begin{align*}
			\int_{\mathcal{M}}\,|\mathcal{K}_t^{\gamma}(f)(x)\,-\,M\,\psi_t^{\gamma}(x,x_{0})|\,\diff{\mu(x)}\,&\leq C(\xi,\gamma,\mathcal{M})\,t^{-\theta_{\gamma}}
			\int_{\mathcal{M}} \psi_t^{\gamma}(x,x_{0})\,\diff{\mu(x)}\int_{B\left( x_{0},\xi\right) }|f(y)|\,\diff{\mu(y)}\\
			&\lesssim t^{-\theta_{\gamma}} \int_{\mathcal{M}} \psi_t^{\gamma}(x,x_{0})\,\diff{\mu(x)}\,\|f\|_{L^1(\mathcal{M})}\lesssim t^{-\theta_{\gamma}},
		\end{align*}
		for $t$ large enough such that 
		$d(x_0,y)\le{\xi}\le t^{1/\gamma}$, where in the last step we used (P2). This proves the desired $L^1(\mathcal{M})$ result.

		We now turn to the proof of the sup norm asymptotics. For $t$ large enough so that 
		$d(x_0,y)\le{\xi}\le t^{1/\gamma}$, by (P3) we get $\psi_t^{\gamma}(x,x_0)\asymp \psi_t^{\gamma}(x,y)$ for all $x\in \mathcal{M}$. Therefore, by \eqref{comment}, (P1) and (P4) we get
		\begin{align*}
			|\mathcal{K}_t^{\gamma}(f)(x)\,-\,M\,\psi_t^{\gamma}(x,x_{0})|&\leq   C(\xi,\gamma,\mathcal{M})\,t^{-\theta_{\gamma}}\int_{B\left( x_{0},\xi\right) }\,\psi_{t}^{\gamma}(x,y)\,|f(y)|\,\diff{\mu(y)}\\
			&\lesssim t^{-\theta_{\gamma}}  \,\sup _{y\in\mathcal{M}}\psi_{t}^{\gamma}(x,y) \,\|f\|_{L^1(\mathcal{M})} \\
			&\lesssim t^{-\theta_{\gamma}} \, V(x,t^{1/\gamma})^{-1},
		\end{align*}
		where in the last step we used (P2). The claim follows.
	\end{proof}

	\subsection{Asymptotics for solutions to the Caffarelli-Silvestre extension problem}
	In this subsection, we study the large-time asymptotic behavior of the family of operators $\{T_t^{\sigma}\}_{t>0}$. To this end, we first give some indispensable estimates concerning the kernels $Q_t^{\sigma}$ and use them to prove that they belong to the class $\mathcal{P}_1$ for all $\sigma\in (0,1)$.
	
	\begin{lemma}\label{aux lemma Q}
		For all $x,y\in \mathcal{M}$, all $t>0$, and all constants $c>0$ and $\kappa>1$, it holds
		$$\int_{0}^{\infty}V(x,\sqrt{u})^{-1}\,e^{-\frac{t^2}{4u}}\,e^{-c\frac{d^2(x,y)}{u}} \, \frac{du}{u^{\kappa}}\asymp \frac{1}{V(x, t+d(x,y))}\,\frac{1}{(t+d(x,y))^{2(\kappa-1)}},$$
		where the implied constants depend only on $c, \kappa$ and $\mathcal{M}$.
		
		In addition, for fixed $r>0$, we have
		$$V(x, t+d(x,y))\,(t+d(x,y))^{2(\kappa-1)}\int_{0}^{r}V(x,\sqrt{u})^{-1}\,e^{-\frac{t^2}{4u}}\,e^{-c\frac{d^2(x,y)}{u}} \, \frac{du}{u^{\kappa}}=\textrm{O}(t^{-N}), \quad \forall N>0$$
		as $t\rightarrow +\infty,$ where the implied constant depends in addition on $r$, $N$.
	\end{lemma}
	
	\begin{proof}
		Set $$I:=\int_{0}^{\infty}V(x,\sqrt{u})^{-1}\,e^{-\frac{t^2}{4u}}\,e^{-c\frac{d^2(x,y)}{u}}\,\frac{du}{u^{\kappa}}.$$ Then, for $C:=\min\{c, 1/4\}$, we have
		\begin{align*}
			I&\leq\int_{0}^{\infty} \, V(x, \sqrt{u})^{-1} e^{-C\frac{t^2+d^2(x,y)}{u}}\,\frac{\diff{u}}{u^{\kappa}}\\
			&\lesssim \frac{1}{(t^2+d^2(x,y))^{\kappa-1}} \int_{0}^{\infty} \,  V\left(x, \sqrt{\frac{t^2+d^2(x,y)}{u}}\right)^{-1} e^{-Cu}\,u^{\kappa-2}\,\diff{u}\\
			&\lesssim \frac{1}{(t^2+d^2(x,y))^{\kappa-1}}\frac{1}{V\left(x, \sqrt{t^2+d^2(x,y)}\right)} \times\\ &\times\left(\int_{0}^{1}e^{-Cu}\,u^{\frac{\nu'}{2}}\,u^{\kappa-2}\,\diff{u}+ \int_{1}^{\infty} e^{-Cu}\,u^{\frac{\nu}{2}}\,u^{\kappa-2}\,\diff{u}\right) \\
			&\lesssim \frac{1}{V(y, t+d(x,y))} \,\frac{1}{ (t+d(x,y))^{2(\kappa-1)}}, 
		\end{align*}  
		where for the last two inequalities we used the volume doubling property \eqref{S2 Comp same center}.
		The lower bound follows similarly.

		Last, as far as asymptotics are concerned, observe that for all $C>0$, we have
		\begin{align*}
			\int_{0}^{r} \, V(y, \sqrt{u})^{-1} e^{-C\frac{t^2+d^2(x,y)}{u}}\,\frac{\diff{u}}{u^{\kappa}}&\lesssim  e^{-\frac{C}{2}\frac{t^2+d^2(x,y)}{r}}\int_{0}^{r} \, V(y, \sqrt{u})^{-1} e^{-\frac{C}{2}\frac{t^2+d^2(x,y)}{u}}\,\frac{\diff{u}}{u^{\kappa}}.
		\end{align*}
		The additional exponential term on the right-hand side allows for the fast decay in $t$ as  $t\rightarrow+\infty$, while for the remaining integral we may conclude as before. 
	\end{proof}	
	
	\begin{corollary}\label{Cor Q}
		There is a constant $C \geq 1$ such that if $d(y,z) \leq t$, then
		$$C^{-1} \leq
		\frac{Q^{\sigma}_t(x,y)}{Q^{\sigma}_t(x,z)}\leq C.$$
	\end{corollary}
	
	\begin{proof}
		Observe first that for any $\gamma>0$, if  $d(y,z)\leq t^{1/\gamma}$, then the triangle inequality implies 
		\begin{equation}\label{ineqTR}
			\frac{1}{2}\leq \frac{t^{1/\gamma}+d(x,y)}{2\,t^{1/\gamma}+d(x,y)}\leq \frac{t^{1/\gamma}+d(x,y)}{t^{1/\gamma}+d(x,z)}\leq \frac{2\,t^{1/\gamma}+d(x,z)}{t^{1/\gamma}+d(x,z)}\leq 2.
		\end{equation}
		Recall next the subordination formula \eqref{kernelCS} for $Q_t^{\sigma}$ and the double-sided heat kernel estimates \eqref{LY}. Then, the claim is a simple consequence of Lemma \ref{aux lemma Q}  for $\kappa=\sigma+1$, which yields
		\begin{equation}\label{kernelQ UL}
			Q_t^{\sigma}(x,y)\asymp \frac{1}{V(x, t+d(x,y))} \,\frac{t^{2\sigma}}{ (t+d(x,y))^{2\sigma}},
		\end{equation}
		the doubling volume condition \eqref{S2 Comp same center} and the inequality \eqref{ineqTR} for $\gamma=1$. 
	\end{proof}
	
	We are now in a position to prove that 
	$$(Q_t^{\sigma})\in \mathcal{P}_1 \quad \text{for all} \quad \sigma\in (0,1).$$ Indeed, as already mentioned, due to the subordination formula \eqref{kernelCS}, the kernel $Q_t^{\sigma}$ satisfies (P1) as well as the first assertion of (P2). The second assertion of (P2) follows immediately from \eqref{kernelQ UL}. (P3) follows by Corollary \ref{Cor Q}. 
	
	Finally, we prove (P4). Fix a basepoint $x_0\in \mathcal{M}$ and consider $y\in \mathcal{M}$ such that $d(x_0,y)< \xi$. 
	
	Write
	\begin{align*}
		|Q_t^{\sigma}(x,y)-Q_t^{\sigma}(x,x_0)|&\leq \frac{t^{2\sigma}}{2^{2\sigma}\Gamma(\sigma)}\int_{0}^{+\infty}|h_u(x,y)-h_u(x,x_0)|\,e^{-\frac{t^2}{4u}}\frac{\diff u}{u^{1+\sigma}}\\
		&:=I_1+I_2, 
	\end{align*}
	where 
	\begin{align*}
		I_1&=\frac{t^{2\sigma}}{2^{2\sigma}\Gamma(\sigma)}\int_{0}^{\xi^2}|h_u(x,y)-h_u(x,x_0)|\,e^{-\frac{t^2}{4u}}\frac{ \diff u}{u^{1+\sigma}}, \\ I_2&=\frac{t^{2\sigma}}{2^{2\sigma}\Gamma(\sigma)}\int_{\xi^2}^{+\infty}|h_u(x,y)-h_u(x,x_0)|\,e^{-\frac{t^2}{4u}}\frac{\diff u}{u^{1+\sigma}}.
	\end{align*}
	Let us first start with $I_2$. Since $u\geq\xi^2> d^2(x_0,y)$, we can use the Hölder estimate \eqref{Holder} for the heat kernel. Therefore, applying Lemma \ref{aux lemma Q} for $\kappa=1+\sigma+\frac{\theta}{2}$, we get
	\begin{align*}
		I_2&\leq C(\sigma, \mathcal{M}) \,\xi^{\theta}\, t^{2\sigma} \int_{\xi^2}^{+\infty}V(x,\sqrt{u})^{-1}\,e^{-\frac{t^2}{4u}}\,e^{-c\frac{d^2(x,x_0)}{u}}\frac{\diff u}{u^{1+\sigma+\frac{\theta}{2}}}\\
		&\leq C(\xi, \sigma, \mathcal{M})\, \frac{1}{V(x, t+d(x,x_0))}\,\frac{t^{2\sigma}}{(t+d(x,x_0))^{2\sigma+\theta}}\\
		&\leq C(\xi, \sigma, \mathcal{M})\, t^{-\theta}\, Q_t^{\sigma}(x,x_0),
	\end{align*}
	where in the last step we used the bounds \eqref{kernelQ UL}.
	It remains to treat $I_1$. For this, observe that by the second part of Lemma \ref{aux lemma Q} for $\kappa=1+\sigma$ and by \eqref{kernelQ UL}, we have for all $x,y\in \mathcal{M}$ and for all $t$ large enough 
	\begin{align*}\frac{t^{2\sigma}}{2^{2\sigma}\Gamma(\sigma)}\int_{0}^{\xi^2}h_u(x,y)\,e^{-\frac{t^2}{4u}}\frac{ \diff u}{u^{1+\sigma}}\lesssim t^{-N}\, Q_t^{\sigma}(x,y), \quad \text{for all} \; N>0.
	\end{align*}
	Therefore, 
	$$I_1\lesssim t^{-N}\,(Q_t^{\sigma}(x,y)+Q_t^{\sigma}(x,x_0)).$$
	Taking $t$ large enough so that $d(x_0,y)<\xi<t$ and making use of Corollary \ref{Cor Q} we finally estimate 
	$$I_1\lesssim t^{-N}\,Q_t^{\sigma}(x,x_0) \qquad \forall N>0.$$
	Altogether,  we get
	\begin{align}\label{Holder Q}
		|Q_t^{\sigma}(x,y)-Q_t^{\sigma}(x,x_0)|\lesssim  t^{-\theta}\,Q_t^{\sigma}(x,x_0) 
	\end{align}
	for all  $x,y,x_0\in \mathcal{M}$ such that $d(x_0,y)<\xi$ and $t$ large enough. This proves (P4) for $\theta_1=\theta$, where $\theta$ is the constant from the heat kernel H\"older inequality \eqref{Holder}.

	\subsection{Asymptotics for solutions to the fractional heat equation} 
	
	As in the previous section, we start by proving upper and lower bounds for the kernel $P_t^{\alpha}$, by using the subordination formula \eqref{kernelFH} and the double-sided estimates of the heat kernel \eqref{LY}.
	
	It is well-known that the subordinator $\eta_{t}^{\alpha}$ cannot be written explicitly, except for the case $\alpha=1$. Let us recall, however, that
	\begin{align} \label{etaUL}
		\eta_{t}^{\alpha}(u) \asymp
		\begin{cases}
			t^{\frac{1}{2-\alpha}}\, u^{-\frac{4-\alpha}{4-2\alpha}}\,e^{-c_{\alpha} t^{\frac{2}{2-\alpha}} u^{-\frac{\alpha}{2-\alpha}}}, &u\leq t^{2/\alpha} \\[5pt]
			t\, u^{-1-\alpha/2}, &u> t^{2/\alpha},
		\end{cases}
	\end{align}
	where $c_{\alpha}=\frac{2-\alpha}{2}\left( \frac{\alpha}{2}\right)^{\frac{\alpha}{2-\alpha}}$, \cite{GS04}. From these, we get that 
	\begin{equation}\label{etaU}
		\eta_{t}^{\alpha}(u) \lesssim	t\, u^{-1-\alpha/2}, \quad t, u>0.
	\end{equation}
	
	\begin{lemma}\label{aux lemma P}
		For all $x,y\in \mathcal{M}$, all $t>0$, and all constants $C>0$, $\kappa\geq 0$, it holds
		$$\int_{0}^{\infty}V(y,\sqrt{u})^{-1}\,\,e^{-C\frac{d^2(x,y)}{u}} \,\eta^{\alpha}_{t}(u)\, \frac{\diff u}{u^{\kappa}}\asymp  \frac{t}{V(x, t^{\frac{1}{\alpha}}+d(x,y))}\,\frac{1}{ (t^{\frac{1}{\alpha}}+d(x,y))^{\alpha+2\kappa}}$$
		where the implied constants depend only on $C, \kappa, \alpha$ and $\mathcal{M}$.
		
		In addition, for fixed $r>0$, we have
		$$V(x, t^{\frac{1}{\alpha}}+d(x,y))\, (t^{\frac{1}{\alpha}}+d(x,y))^{\alpha+2\kappa}\int_{0}^{r}V(y,\sqrt{u})^{-1}\,e^{-C\frac{d^2(x,y)}{u}} \,\eta_t^{\alpha}(u)\, \frac{du}{u^{\kappa}}=\textrm{O}(t^{-N}) \quad \forall N>0$$
		as $t\rightarrow +\infty,$ where the implied constant depends in addition on $r$, $N$.
	\end{lemma}
	
	\begin{proof} Write 
		$$K=\int_{0}^{\infty}V(y,\sqrt{u})^{-1}\,e^{-C\frac{d^2(x,y)}{u}} \,\eta_t^{\alpha}(u)\, \frac{\diff u}{u^{\kappa}}.$$
		To estimate the integral $K$, we distinguish cases.
		
		\textit{Case I: $d(x,y)\geq t^{1/\alpha}$.} Then, $$\frac{1}{4}(t^{\frac{1}{\alpha}}+d(x,y))^2\leq d^2(x,y)\leq  (t^{\frac{1}{\alpha}}+d(x,y))^2.$$
		Therefore, by   \eqref{etaU}, we get that 
		\begin{align*}
			K&\lesssim t \int_{0}^{\infty}u^{-1-\frac{\alpha}{2}-\kappa} \, V(x, \sqrt{u})^{-1} \exp\left\{{-\frac{C}{4}\, \left(\frac{t^{\frac{1}{\alpha}}+d(x,y)}{\sqrt{u}}\right)^2}\right\}\,\diff{u}\\
			&\lesssim \frac{t}{V(x, t^{\frac{1}{\alpha}}+d(x,y))} \int_{0}^{\infty}u^{-1-\frac{\alpha}{2}-\kappa} \,  \exp\left\{-\frac{C}{8}\, \left(\frac{t^{\frac{1}{\alpha}}+d(x,y)}{\sqrt{u}}\right)^2\right\}\,\diff{u}\\
			&\lesssim \frac{t}{V(x, t^{\frac{1}{\alpha}}+d(x,y))}\,\frac{1}{ (t^{\frac{1}{\alpha}}+d(x,y))^{\alpha+2\kappa}}.
		\end{align*}  
		Here, for the second inequality we used the volume doubling property \eqref{S2 Comp same center} and the fact that $x^{\mu}\,e^{-c\,x^2}\leq \textrm{const.}(c,\mu)\, e^{-\frac{c}{2}\,x^2}$ for all $c, \mu>0$ and $x\in (0,+\infty)$. For the last inequality we applied a change of variables. 
		
		For a lower bound, by \eqref{etaUL} we get for all $C>0$ that 
		\begin{equation*}
			K\geq\int_{t^{2/\alpha}}^{\infty}V(y,\sqrt{u})^{-1}\,e^{-C\frac{d^2(x,y)}{u}} \,\eta_t^{\alpha}(u)\, \frac{\diff u}{u^{\kappa}} \gtrsim t \int_{(t^{\frac{1}{\alpha}}+d(x,y))^{2}}^{2\,(t^{\frac{1}{\alpha}}+d(x,y))^{2}}u^{-1-\frac{\alpha}{2}-\kappa} \, V(x, \sqrt{u})^{-1}\,e^{-C\frac{d^2(x,y)}{u}}\diff{u}.
		\end{equation*} 
		Taking into account that $u\asymp (t^{\frac{1}{\alpha}}+d(x,y))^{2} \asymp d^2(x,y)$, the claim follows immediately by the volume doubling property.

		\textit{Case II: $d(x,y)< t^{1/\alpha}$.} Then, $t^{\frac{2}{\alpha}}\asymp t^{\frac{2}{\alpha}}+d^2(x,y) \asymp (t^{\frac{1}{\alpha}}+d(x,y))^2,$ so 
		\begin{equation}\label{sim}
			\frac{t}{(t^{\frac{1}{\alpha}}+d(x,y))^{\alpha}}\asymp 1.
		\end{equation}
		Write 
		$$I+J:=\int_0^{t^{
				2/\alpha}}V(y,\sqrt{u})^{-1}\,e^{-C\frac{d^2(x,y)}{u}} \,\eta_t^{\alpha}(u)\, \frac{\diff u}{u^{\kappa}}+ \int_{t^{
				2/\alpha}}^{\infty}V(y,\sqrt{u})^{-1}\,e^{-C\frac{d^2(x,y)}{u}} \,\eta_t^{\alpha}(u)\, \frac{\diff u}{u^{\kappa}}.$$
		
		We first deal with $I$. By \eqref{etaUL}, we have
		\begin{align*}
			I&\asymp	t^{\frac{1}{2-\alpha}}	\int_{0}^{
				t^{2/\alpha}} u^{-\frac{4-\alpha}{4-2\alpha}-\kappa}\,e^{-c_{\alpha}\left(\frac{t^{2/\alpha}}{u}\right)^{\frac{\alpha}{2-\alpha}}}\;V(x, \sqrt{u})^{-1}\;e^{-C\,\frac{d^2(x,y)}{u}}\;\diff{u} \\
			&\asymp t^{-2\kappa/\alpha}	\int_{1}^{\infty} u^{\frac{3\alpha-4}{4-2\alpha}+\kappa}\,e^{-c_{\alpha} u^{\frac{\alpha}{2-\alpha}}}\;V\left(x, \frac{t^{1/\alpha}}{\sqrt{u}}\right)^{-1}\;e^{-C\,\frac{d^2(x,y)}{t^{2/\alpha}}u}\;\diff{u}\\
			&\lesssim t^{-2\kappa/\alpha}\,V(x, t^{1/\alpha})^{-1}\int_{1}^{\infty} u^{\frac{3\alpha-4}{4-2\alpha}+\kappa}\,e^{-c_{\alpha} u^{\frac{\alpha}{2-\alpha}}}\;u^{\frac{\nu'}{2}}\, \diff{u}  \\
			&\lesssim \frac{t}{V(x, t^{\frac{1}{\alpha}}+d(x,y))}\,\frac{1}{   (t^{\frac{1}{\alpha}}+d(x,y))^{\alpha+2\kappa}},
		\end{align*}
		where we first performed a change of variables, while for the last two inequalities we used \eqref{S2 Comp same center} and \eqref{sim}. The lower bound for $I$ is proved similarly, using that $\exp\left\{-C\,\frac{d^2(x,y)}{t^{2/\alpha}}u\right\}>\exp\left\{-C\,u\right\}$.

		We now turn to $J$. We restrict to proving upper bounds, since the proof for lower bounds runs similarly. Observe first that for any $C>0$,  we have 
		\begin{equation}\label{GUL}
			\exp\left\{-C\,\frac{d^2(x,y)}{t^{2/\alpha}}u\right\}\asymp 1, \quad \text{ for } \;  d(x,y)<t^{1/\alpha}, \; u\in (0,1).
		\end{equation}
		Then, by \eqref{etaUL}, the change of variables $u\mapsto t^{2/\alpha}/u$, and \eqref{GUL}  we get
		\begin{align*}
			J&\asymp	
			t	\int_{t^{2/\alpha}}^{\infty} u^{-1-\frac{\alpha}{2}-\kappa}  \;V(x, \sqrt{u})^{-1}\;e^{-C\,\frac{d^2(x,y)}{u}}\;\diff{u} \asymp	\int_{0}^{1} u^{-1+\frac{\alpha}{2}+\kappa}   \;V\left(x, \frac{t^{1/\alpha}}{\sqrt{u}}\right)^{-1}\;\diff{u}\\
			&\lesssim t^{-2\kappa/\alpha}\, V(x, t^{1/\alpha})^{-1}\int_{0}^{1} u^{-1+\frac{\alpha}{2}+\kappa} \;u^{\frac{\nu}{2} }\;\diff{u}\\
			&\lesssim\frac{t}{V(x, t^{\frac{1}{\alpha}}+d(x,y))}\,\frac{1}{ (t^{\frac{1}{\alpha}}+d(x,y))^{\alpha+2\kappa}},
		\end{align*}
		using the volume doubling property \eqref{S2 Comp same center}. This concludes the proof as far as bounds are concerned. 
		
		Last, for asymptotics, we distinguish cases once again. If $d(x,y)\geq t^{1/\alpha}$, the claim follows by using \eqref{etaU}, by observing that for all $C>0$, 
		\begin{align*}
			\int_{0}^{r}V(y,\sqrt{u})^{-1}e^{-C\frac{d^2(x,y)}{u}} \eta_t^{\alpha}(u) \frac{\diff u}{u^{\kappa}}
			&\lesssim \exp\left\{-\frac{C}{8} \left(\frac{t^{\frac{1}{\alpha}}+d(x,y)}{\sqrt{r}}\right)^2\right\}\times\\ & \times t \int_{0}^{\infty}u^{-1-\frac{\alpha}{2}-\kappa}  V(x, \sqrt{u})^{-1} \exp\left\{{-\frac{C}{8} \left(\frac{t^{\frac{1}{\alpha}}+d(x,y)}{\sqrt{u}}\right)^2}\right\}\,\diff{u}
		\end{align*}  
		and by resuming the computations for $K$. Therefore, the exponential term allows for the claimed fast decay in time. If $d(x,y)<t^{1/\alpha}$, take $t$ large enough so that $r<t^{2/\alpha}$ whence the claim follows by \eqref{etaUL}, by resuming the computations for $I$ and  by observing that for all $C>0$,
		\begin{align*}
			\int_{0}^{r}V(y,\sqrt{u})^{-1}\,\,e^{-C\frac{d^2(x,y)}{u}} \,\eta_t^{\alpha}(u)\, \frac{\diff u}{u^{\kappa}}
			&\lesssim \exp\left\{-\frac{c_{\alpha}}{2} \left(\frac{t^{2/\alpha}}{r}\right)^{\frac{\alpha}{2-\alpha}}\right\} \times \\
			&\times	t^{\frac{1}{2-\alpha}}	\int_{0}^{
				r} u^{-\frac{4-\alpha}{4-2\alpha}-\kappa}\,e^{-\frac{c_{\alpha}}{2} \left(\frac{t^{2/\alpha}}{u}\right)^{\frac{\alpha}{2-\alpha}}}V(x, \sqrt{u})^{-1}e^{-C\,\frac{d^2(x,y)}{u}}\diff{u}.
		\end{align*}
		The proof is now complete.
	\end{proof}
	
	\textbf{Remark.} For $\kappa=0$, the estimates of Lemma \ref{aux lemma P} amount to bounds for the kernel of the fractional heat semigroup (see the next corollary), which were already known in the literature: we refer to \cite[Eq (3.5)]{Shi17} where such bounds are actually established on a more general setting. It is mostly the case $\kappa>0$ that will be of interest to us.
	
	\begin{corollary}\label{Cor P}
		There is a constant $C\geq 1$ such that if $d(y,z) \leq t^{1/\alpha}$, then
		$$C^{-1} \leq
		\frac{P^{\alpha}_t(x,y)}{P^{\alpha}_t(x,z)}\leq C.$$
	\end{corollary}
	\begin{proof}
		Recall first the subordination formula \eqref{kernelFH} and the double-sided heat kernel estimates \eqref{LY}. Then, by Lemma \ref{aux lemma P} for $\kappa=0$ we get for all $x, y\in \mathcal{M}$ and all $t>0$ that
		\begin{equation}\label{kernelP UL}
			P_t^{\alpha}(x,y)\asymp \frac{1}{V(x, t^{\frac{1}{\alpha}}+d(x,y))} \,\frac{t}{ (t^{\frac{1}{\alpha}}+d(x,y))^{\alpha}}
		\end{equation}
		Then the claim follows as in Corollary \ref{Cor Q}.
	\end{proof}
	We are now in a position to prove that 
	$$(P_t^{\alpha})\in \mathcal{P}_{\alpha}, \quad \alpha\in (0,2).$$ Indeed, as already mentioned, due to the subordination formula \eqref{kernelFH}, the kernel $P_t^{\alpha}$ satisfies (P1) as well as the first assertion of (P2). The second assertion of (P2) follows immediately from \eqref{kernelP UL}. (P3) follows by Corollary \ref{Cor P}. 
	
	Finally, we prove (P4). Fix a basepoint $x_0\in \mathcal{M}$ and consider $y\in \mathcal{M}$ such that $d(x_0,y)< \xi$. Write
	\begin{align*}
		|P_t^{\alpha}(x,y)-P_t^{\alpha}(x,x_0)|&\leq \int_{0}^{+\infty}|h_u(x,y)-h_u(x,x_0)|\,\eta_{t}^{\alpha}(u)\,\diff u\\
		&:=I_1+I_2, 
	\end{align*}
	where 
	\begin{align*}
		I_1&=\int_{0}^{\xi^2}|h_u(x,y)-h_u(x,x_0)|\,\eta_{t}^{\alpha}(u)\,\diff u, \\ I_2&=\int_{\xi^2}^{+\infty}|h_u(x,y)-h_u(x,x_0)|\,\eta_{t}^{\alpha}(u)\,\diff u.
	\end{align*}
	Let us first start with $I_2$. Since $u\geq\xi^2> d^2(x_0,y)$, we can use the Hölder estimate \eqref{Holder} for the heat kernel. Therefore, applying Lemma \ref{aux lemma P} for $\kappa=\theta/2$, we get
	\begin{align*}
		I_2&\leq C(\alpha, \mathcal{M}) \,\xi^{\theta}\,\int_{0}^{\infty}V(y,\sqrt{u})^{-1}\,\,e^{-c\frac{d^2(x,y)}{u}} \,\eta^{\alpha}_{t}(u)\, \frac{\diff u}{u^{\theta/2}}\\
		&\leq C(\xi, \alpha, \mathcal{M}) \, \frac{t}{V(x, t^{\frac{1}{\alpha}}+d(x,y))}\,\frac{1}{ (t^{\frac{1}{\alpha}}+d(x,y))^{\alpha+\theta}} \\
		&\leq C(\xi, \alpha, \mathcal{M}) \, t^{-\theta/ \alpha}\, P_t^{\alpha}(x,x_0),
	\end{align*}
	where in the last step we used the bounds \eqref{kernelP UL}.
	It remains to treat $I_1$. For this, observe that by the second part of Lemma \ref{aux lemma P} for $\kappa=0$ and by \eqref{kernelP UL}, we have for all $x,y\in \mathcal{M}$ and for all $t$ large enough,
	\begin{align*}\int_{0}^{\xi^2}h_u(x,y)\,\eta^{\alpha}_{t}(u)\, \diff u\lesssim t^{-N}\, P_t^{\alpha}(x,y), \quad \text{for all} \quad N>0.
	\end{align*}
	Therefore, 
	$$I_1\lesssim t^{-N}\,(P_t^{\alpha}(x,y)+P_t^{\alpha}(x,x_0)).$$
	Taking $t$ large enough so that $d(x_0,y)<\xi<t^{1/\alpha}$ and making use of Corollary \ref{Cor P} we finally estimate 
	$$I_1\lesssim t^{-N}\,P_t^{\alpha}(x,x_0) \qquad \forall N>0, \, \forall t>\xi^{\alpha}.$$
	Altogether,  we get
	\begin{align}\label{Holder P}
		|P_t^{\alpha}(x,y)-P_t^{\alpha}(x,x_0)|\lesssim  t^{-\theta/\alpha}\,P_t^{\alpha}(x,x_0) 
	\end{align}
	for all $t$ large enough and all $x,y,x_0\in \mathcal{M}$ such that $d(x_0,y)<\xi$. This proves (P4) for $\theta_{\alpha}=\theta/\alpha$, where $\theta$ is the constant from the heat kernel H\"older inequality \eqref{Holder}.

	\vspace*{0.5cm}

	\subsubsection{Final remarks on the rate of convergence.} The rate of convergence for continuous and compactly supported initial data is optimal, both for the extension problem ($\textrm{O}(t^{-\theta})$ from \eqref{Holder Q}) and for the fractional heat equation ($\textrm{O}(t^{-\theta/\alpha})$ from \eqref{Holder P}), in the following sense: on euclidean space, the heat kernel H{\"o}lder inequality \eqref{Holder} holds for $\theta=1$, and it is known that, concerning the fractional heat equation, the optimal rate for convergence in the $L^1$ norm for compactly supported initial data is $\textrm{O}(t^{-1/\alpha})$,  \cite[Theorem 3.2]{Vaz2018}.

	Moreover, instead of using (P4) for $P_t^{\alpha}$, one could pursue using a H{\"older} continuity estimate for $P^{\alpha}_t$, that is, that there is a constant $\Theta>0$ such that
	$$|P_t^{\alpha} (x, y) - P_t^{\alpha} (x, z)| \lesssim \left( \frac{d(y,z)}{t^{1/\alpha}}\right)^{\Theta}P_t^{\alpha} (x, y),$$
	when $d(y, z) \leq t^{1/\alpha}$. For a proof, see \cite[Theorem 4.14]{CK03} for stable-like processes on a rather general setting (alternatively, one can modify for all $\alpha \in (0,2)$ the result of \cite[Theorem 4]{DzPr2018} for the Poisson operator, using the semigroup property, the fact that the fractional heat kernel is a probability measure and Corollary \ref{Cor P}). However, for compactly supported initial data, this would imply $L^1(\mathcal{M})$ convergence at speed  $\textrm{O}(t^{-\Theta/\alpha})$. In this sense, our approach gives more information on the rate of convergence for this class of data. Notice moreover that as far as the extension problem is concerned, for $\sigma\neq 1/2$, the operators $\{T_{t}^{\sigma}\}_{t>0}$ do not form a semigroup. Indeed, observe that 
	$$\frac{t^{2\sigma}}{2^{2\sigma}\Gamma(\sigma)}\int_{0}^{+\infty}e^{-u
		\lambda^2}\,e^{-\frac{t^2}{4u}}\frac{du}{u^{1+\sigma}}=\frac{t^{\sigma}}{2^{\sigma-1}\Gamma(\sigma)}\,K_{\sigma}(t\lambda)\lambda^{\sigma}, \quad \lambda >0,\quad \sigma \in (0,1), $$
	where $K_{\sigma}(\cdot)$ is the modified Bessel function of the second kind and index $\sigma$ (recall that for $\sigma=1/2$, we have $K_{1/2}(x)=\sqrt{\frac{\pi}{2x}}e^{-x}$, so the right hand side above becomes $e^{-t\lambda}$).
	
	Finally, following some ideas from the euclidean setting in \cite{Vaz2017, Vaz2018}, let us show how one can prescribe any rate of convergence to solutions of the fractional heat equation by choosing appropriate initial data (the proof works also for $T_t^{\sigma}$, with obvious modifications).
	More precisely, we shall show that given any
	decreasing and positive function $\phi(t)$ such that $\phi(t)\rightarrow 0$ as $t\rightarrow+\infty$, there is a solution $w$ with mass $M = 1$ satisfying
	\begin{equation}\label{norate}
		\,\Vert \left\vert w(t,\,.\,)-P_t^{\alpha}(\,.\,,x_{0})\right\vert V(\,.\,,
		t^{1/\alpha})\Vert _{L^{\infty }(\mathcal{M})} \gtrsim k\phi(t_k),
	\end{equation}
	for a sequence of times $t_k\rightarrow +\infty$ that can be chosen.

	To prove \eqref{norate}, fix a basepoint $x_0\in \mathcal{M}$. Let $(m_k)_{k\geq 1}$ be a nonnegative summable sequence with $\sum_{k=1}^{\infty}m_k=\epsilon<1$, and consider initial data $(1-\epsilon)\,\delta_{x_0}(x)+\sum_{k=1}^{\infty}m_k \,\delta_{x_k}(x)$. Observe that the total mass is $1$. Also, the points $x_k\in \mathcal{M}$, $k\geq 1$, where the weighted Dirac measures are located are such that $r_k:=d(x_0,x_k)\rightarrow +\infty$.  In this case, the action of $W_t^{\alpha}$ yields the following solution of the fractional heat equation,
	\begin{align*}
		w(t,x)=(1-\epsilon)\,P_t^{\alpha}(x,x_0)+\sum_{k=1}^{\infty}m_k\,P_t^{\alpha}(x,x_k).
	\end{align*}
	Therefore at $x=x_0$ we have
	\begin{align*}
		|w(t,x_0)-P_t^{\alpha}(x_0,x_0)|&=\left| \sum_{k=1}^{\infty}m_k\, (P_t^{\alpha}(x_0,x_k)-P_t^{\alpha}(x_0,x_0)) \right|\\
		&=P_t^{\alpha}(x_0,x_0) \left| \sum_{k=1}^{\infty}m_k\, \left( \frac{P_t^{\alpha}(x_0,x_k)}{P_t^{\alpha}(x_0,x_0)}-1\right) \right|\\
		&\geq c_1\, V(x_0, t^{1/\alpha})^{-1}\, \left| \sum_{k=1}^{\infty}m_k\, \left( \frac{P_t^{\alpha}(x_0,x_k)}{P_t^{\alpha}(x_0,x_0)}-1\right) \right|
	\end{align*}	
	for some constant $c_1>0$ due to \eqref{kernelP UL}. Now, again due to \eqref{kernelP UL}, there is a constant $C_2\geq 1$ such that 
	\begin{align*}
		C_2^{-1}\,\frac{V(x_0, t^{1/\alpha })}{V(x_0, t^{1/\alpha}+r_k)} \,\frac{t}{ (t^{1/\alpha}+r_k)^{\alpha}}\leq \frac{P_t^{\alpha}(x_0,x_k)}{P_t^{\alpha}(x_0,x_0)}\leq C_2\,\frac{V(x_0, t^{1/\alpha})}{V(x_0, t^{1/\alpha}+r_k)} \,\frac{t}{ (t^{1/\alpha}+r_k)^{\alpha}},
	\end{align*}	
	where
	$$\frac{V(x_0, t^{1/\alpha})}{V(x_0, t^{1/\alpha}+r_k)}\leq C\, \left( 1+\frac{r_k}{t^{1/\alpha} } \right)^{-\nu'}, \quad \frac{t}{ (t^{1/\alpha}+r_k)^{\alpha}}= \left( 1+\frac{r_k}{t^{1/\alpha} } \right)^{-\alpha}$$
	for some $C>1$ owing to \eqref{S2 Comp same center}.  
	Consider now $\phi(t) \searrow 0$ as $t\rightarrow +\infty$, and choose iteratively $t_k$ and $x_k$ as follows:
	given choices for the steps $1, 2, . . . , k-1$, pick $t_k$ to be much larger than $t_{k-1}$ and such that
	$\phi(t_k) \leq  m_k/(2k)$. This is possible since $\phi(t)$ decreases to zero. Choose now $d(x_k,x_0) = r_k$ so large that $\left( 1+\frac{r_k}{t_k^{1/\alpha} } \right)^{-\nu'-\alpha}<1/(2\,C\,C_2)$. Therefore
	\begin{align*}
		|w(t,x_0)-P_t^{\alpha}(x_0,x_0)|\,V(x_0, t^{1/\alpha})\geq c_1\,k\phi(t_k),
	\end{align*}
	which proves the claim.

	\section{The Poisson semigroup on rank one non-compact symmetric spaces}
	This section deals with the Poisson semigroup convergence on rank one non-compact symmetric spaces. Our aim is to show that in this case, non-euclidean phenomena occur, namely, the convergence results of the previous sections fail. More precisely, our aim is to prove Theorem \ref{Main Poisson}.

	The Poisson semigroup $e^{-t\sqrt{-\Delta}}$ has
	been studied in various settings, including hyperbolic space  (and more generally, non-compact symmetric spaces), see for instance \cite{AJ99, CowGM} and the references therein. Information for its kernel $p_t$ can be deduced by its subordination to the heat kernel. Therefore, in the rank one case, by well-known properties of the heat kernel, the Poisson kernel $p_t$ is a radial positive function, and for $f\in L^1(\mathbb{X})$, we may write
	$$e^{-t\sqrt{-\Delta}}f(x)=e^{-t\sqrt{-\Delta}}f(gK)=\int_{\mathbb{X}}p_t(d(x,y))f(y)\,\diff{\mu(y)}=\int_{G}p_t(h^{-1}g)f(h)\,\diff{h}.$$

	We next recall some results about its large time behavior, \cite[Theorems 4.3.1 and 5.3.1]{AJ99}. Notice the exponential decay in time and space, which demonstrates the effect of geometry.
	
	\begin{theorem}\cite{AJ99} The Poisson kernel on $\mathbb{X}$ satisfies 
		\begin{equation}\label{bounds}
			p_t(r)\asymp t\, (1+r)\, (t^2+r^2)^{-\frac{5}{4}}\,e^{-\rho r-\rho \sqrt{t^2+r^2}}, \quad \text{ for $t$ large}.
		\end{equation}
		In addition, we have
		\begin{equation}\label{asymp}
			p_t(r)\sim 2^{m_{2\alpha}}\pi^{-\frac{n}{2}}\rho^{\frac{3}{2}}\, \gamma\left(\rho\frac{r}{\sqrt{t^2+r^2}}\right)t\, r\, (t^2+r^2)^{-\frac{5}{4}}\,e^{-\rho r-\rho \sqrt{t^2+r^2}}, \quad \text{ as } t\rightarrow +\infty,
		\end{equation}
		where  $$\gamma(s)=\frac{\Gamma(s+\frac{m_{\alpha}}{2})\Gamma(\frac{s}{2}+\frac{\rho}{2})}{\Gamma(s+1)\Gamma(\frac{s}{2}+\frac{m_{\alpha}}{4})}, \quad s\geq 0.$$
	\end{theorem}
	
	Recall that the Poisson kernel has total mass $1$. We next determine its \textit{critical region}, that is, a region $\Omega_t\subseteq \mathbb{X}$ such that 
	$$\int_{\mathbb{X}\smallsetminus \Omega_{t}}p_t(x)\,\diff{\mu}(x)\longrightarrow 0 \quad \text{as} \quad t\rightarrow +\infty$$
	or equivalently,  
	$$\int_{\Omega_{t}}p_t(x)\,\diff{\mu}(x)\longrightarrow 1 \quad \text{as} \quad t\rightarrow +\infty.$$
	This notion will be substantial for our purposes.
	
	\begin{proposition}\label{Poisson critical}
		Let $0<\epsilon<2$. Then the critical region for the Poisson kernel on $\mathbb{X}$ is 
		$$\Omega_t=\{x\in \mathbb{X}:\; t^{2-\epsilon}\leq d(x,o) \leq t^{2+\epsilon}\},$$
		for $t$ large. More precisely, 
		$$	\int_{\mathbb{X}\,\smallsetminus\,\Omega_{t}}\,p_t(x)\,\diff{\mu}(x)=\textrm{O}(t^{-\frac{\epsilon}{2}}).$$
	\end{proposition}
	\begin{proof}
		Using the bounds \eqref{bounds}, the radiality of the Poisson kernel and the integration formula \eqref{vol} along with \eqref{dens}, we have, for $b>a\geq 0$,
		\begin{align*}
			\int_{a\leq d(x,o)\leq b} p_t(x)\,\diff{\mu}(x)\lesssim\int_{a\leq r\leq b} t\, (1+r)\, (t^2+r^2)^{-\frac{5}{4}}\,e^{-\rho(\sqrt{t^2+r^2}-r)}\,dr.
		\end{align*}
		On the one hand, we compute 
		$$\int_{ d(x,o)< t^{2-\epsilon}}p_t(x)\,\diff{\mu}(x)\leq C(N, \epsilon)\, t^{-N}\quad \forall N>0,$$ 
		due to the fact that in this case, for $t$ large enough we have $\sqrt{t^2+r^2}\leq 2r$, so $$\exp\left(-\rho(\sqrt{t^2+r^2}-r)\right)=\exp\left(-\rho\frac{t^2}{\sqrt{t^2+r^2}+r}\right)\leq \exp\left(-\frac{\rho}{3} t^{\epsilon}\right).$$
		On the other hand,  we compute
		\begin{align*}
			\int_{ d(x,o)> t^{2+\epsilon}}p_t(x)\,\diff\mu(x)\lesssim \int_{ r\geq t^{2+\epsilon}} t\, (1+r)\, (t^2+r^2)^{-\frac{5}{4}}\,\diff{r}
			\lesssim\int_{ r\geq t^{2+\epsilon}} t\, r^{-\frac{3}{2}}\,\diff{r} \lesssim t^{-\frac{\epsilon}{2}},
		\end{align*}
		which completes the proof.
	\end{proof}
	We next give a lemma related to the Busemann function on $\mathbb{X}$. 
	
	\begin{lemma}\label{Busemann} Let $y=y_0K$ be in a bounded region of $\mathbb{X}$. Then, for every $x=gK$ in the critical region $\Omega_t$,
		$$d(x,o)-d(x,y)
		=\tau(k^{-1}y_0)
		+\textnormal{O}\big(t^{-2+\epsilon}\big), $$
		Here,  $k$ is the left component of $g$ in the Cartan decomposition and $\exp(\tau(k^{-1}y_0)H_0)$ is the middle component of $k^{-1}y_0$ in the Iwasawa decomposition. 
	\end{lemma}

	\begin{proof} The arguments follow closely those of  \cite[Lemma 3.8]{APZ2023}, but we include the proof for the sake of completeness. Write $x=gK$, where $g=k\,(\exp rH_0)\,k'$ in the Cartan decomposition. Then $d(gK, o)=r$. Consider the Iwasawa decomposition
		$k^{-1}y_0=n(k^{-1}y_0)\,(\exp{\tau(k^{-1}y_0)\, H_0})\,k''$ for some $k''\in{K}$. Then,
		\begin{align*}\label{S3 distance decomposition in lemma}
			d(x,y)=d(gK,y_0K)\,
			&=\,d\big(k\,(\exp r H_0)\,K,\;k\,n(k^{-1}y_0)\,(\exp{\tau(k^{-1}y_0)H_0})\,K\big)
			\notag\\[5pt]
			&=\,d\big(\exp(-rH_0)\,[n(k^{-1}y_0)]^{-1}(\exp{rH_0})\,K,\;
			\exp((\tau(k^{-1}y_0)-r)H_0)\,K\big),
		\end{align*}
		therefore we write
		\begin{align*}
			d(x,o)-d(x,y)=d(gK,o)-d(gK,y_0K)\,
			&=\,\overbrace{\vphantom{\Big|}
				d(gK,o)-d\big(\exp(\tau(k^{-1}y_0)-r)H_0)\,K,o\big)}^{I}\\
			&+\,\underbrace{\vphantom{\Big|}
				d\big(\exp((\tau(k^{-1}y_0)-r)H_0)\,K,o\big)-d(gK,y_0K)}_{II}.
		\end{align*}
		On the one hand, we have
		\begin{align*}
			I\,
			=\,r-|\tau(k^{-1}y_0)-r |\,
			&=\,\frac{2\,r  \,\tau(k^{-1}y_0)-\tau(k^{-1}y_0)^{2}}{
				r+|\tau(k^{-1}y_0)-r |}\\[5pt]
			&=\,\tau(k^{-1}y_0)
			+\textrm{O}\big(\tfrac{1}{r}\big)\\[5pt]
			&=\,\tau(k^{-1}y_0)
			+\textrm{O}\big({t^{-2+\epsilon}}\big)
		\end{align*}
		by using that $r\geq t^{2-\epsilon}$ and the well-known fact that $|\tau(k^{-1}y_0)|\leq d(k^{-1}y_0K, o)=d(y, o)$, thus bounded.
		On the other hand, the term $II$  tends exponentially fast to $0$, see for instance \cite[Lemma 3.8]{APZ2023}, thus we are done.
	\end{proof}
	
	The next lemma is crucial for our proof.
	
	\begin{lemma}\label{quotient} 
		Let $x=k(\exp rH_0)K \in \Omega_t$, $y=y_0K$ be bounded  and let $0<\epsilon<1/2$. Then, 
		$$\frac{p_t(x,y)}{p_t(x,o)}=e^{2\rho\, \tau(k^{-1}y_0)}+\textrm{O}(t^{-2+4\epsilon}).$$
	\end{lemma}
	
	\begin{proof} Write $r=d(x,o)$, $s=d(x,y)$ and let $d(y,o)<\xi$, for some $\xi>0$. By the triangle inequality, we have $|r-s|\leq \xi$ and since $x$ is in the critical region, we have $t^{2-\epsilon}\leq r\leq t^{2+\epsilon}$. In addition, for $t$ large enough, we have
		$$\frac{1}{2}\,t^{2-\epsilon}\leq t^{2-\epsilon}-\xi\leq s=d(x,y)\leq t^{2+\epsilon}+\xi\leq 2\,t^{2+\epsilon},$$
		and $t^2+r^2\asymp r^2$, $t^2+s^2\asymp s^2$.
		
		By \eqref{asymp}, we get
		\begin{align*} \frac{p_t(d(x,y))}{p_t(d(x,0))}\sim   \frac{\gamma\left(\rho\frac{s}{\sqrt{t^2+s^2}}\right)}{\gamma\left(\rho\frac{r}{\sqrt{t^2+r^2}}\right)}\, \frac{s}{r}\,\frac{(t^2+r^2)^{\frac{5}{4}}}{(t^2+s^2)^{\frac{5}{4}}} \exp\left\{\rho(r-s)\left(1+\frac{r+s}{\sqrt{t^2+r^2}+\sqrt{t^2+s^2}}\right)\right\}.
		\end{align*}
		We next give asymptotics for the terms of the quotient on the right hand side. 
		
		First, recall by \cite[p.1042]{AJ99} that the function $\gamma$ satisfies
		\begin{equation}\label{b func}
			\gamma(u) \asymp 1, \quad \frac{d}{du}\gamma(u)=O(1),
		\end{equation}
		if $u>0$ is bounded above, and below away from zero.	Therefore, by the mean value theorem and \eqref{b func}, we have, for some $r_0$ between $r, s>0$,
		\begin{align*} 
			\left|	\gamma\left(\rho\frac{r}{\sqrt{t^2+r^2}}\right)-\gamma\left(\rho\frac{s}{\sqrt{t^2+s^2}}\right) \right| &\lesssim |r-s|  \left| \gamma'\left(\rho\frac{r_0}{\sqrt{t^2+r_0^2}}\right)\right|\frac{t^2}{(t^2+r_0^2)^{3/2}}\\
			&\lesssim t^2/r_0^{3}.
		\end{align*} 
		Therefore, again by \eqref{b func} and by the fact that $r_0\gtrsim t^{2-\epsilon}$, we have
		$$ \gamma\left(\rho\frac{s}{\sqrt{t^2+s^2}}\right) /\gamma\left(\rho\frac{r}{\sqrt{t^2+r^2}}\right) =1+ O\left({t^{-4+3\epsilon}}\right).$$ 
		Next, by a similar mean value argument applied to $(t^2+(.)^2)^{\frac{5}{4}}$, we  have
		$$\frac{(t^2+r^2)^{\frac{5}{4}}}{(t^2+s^2)^{\frac{5}{4}}}=
		1+O\left({t^{-2+4\epsilon}}\right).$$ 
		Also, since $|r-s|\leq \xi$, we have $$\frac{s}{r}=1+O\left({t^{-2+\epsilon}}\right).$$ 
		
		It remains to deal with the exponential terms, which are the main ones. We first claim that 
		\begin{equation}\label{exp2}
			\frac{r+s}{\sqrt{t^2+r^2}+\sqrt{t^2+s^2}}=1+O\left({t^{-2+2\epsilon}}\right).
		\end{equation}
		Indeed, consider the function $f(u)=\sqrt{u^2+r^2}+\sqrt{u^2+s^2}$, $u \geq 0$, and observe that the left hand side of \eqref{exp2} is equal to $f(0)/f(t)$. Then, the mean value theorem for $f$ in $[0,t]$  together with the fact that  $$f'(u)\lesssim  \frac{u}{r}+\frac{u}{s}\lesssim t^{-1+\varepsilon}, \qquad f(u)\gtrsim  t^{2-\varepsilon}, \qquad \forall u\in[0,t],$$   yield the claimed asymptotics \eqref{exp2}.
		Finally, in Lemma \ref{Busemann} it was shown that
		$$r-s=d(x,o)-d(x,y)=d(gK, o)-d(gK, y_0K)
		=\tau(k^{-1}y_0)
		+\textnormal{O}\big(t^{-2+\epsilon}\big).$$
		Therefore, \begin{align*}
			\exp\left\{\rho(r-s)\left(1+\frac{r+s}{\sqrt{t^2+r^2}+\sqrt{t^2+s^2}}\right)\right\}&=e^{2\rho\, \tau(k^{-1}y_0)+\textrm{O}(t^{-2+2\epsilon})}=e^{2\rho\, \tau(k^{-1}y_0)}+\textrm{O}(t^{-2+2\epsilon}).
		\end{align*}
		
		Altogether, we have 
		$$\frac{p_t(x,y)}{p_t(x,o)}=e^{2\rho\, \tau(k^{-1}y_0)}+\textrm{O}(t^{-2+4\epsilon}).$$
	\end{proof}
	
	\begin{proof}[Proof of Theorem \ref{Main Poisson}] We consider the case of continuous compactly supported initial data. Next, we work separately outside and inside the critical region: we show first that $\|e^{-t\sqrt{-\Delta}}f\|_{L^1(\mathbb{X} \smallsetminus \Omega_t)}\rightarrow 0$  for all $f\in \mathcal{C}_c(\mathbb{X})$, without any further symmetry assumptions on $f$. However, the convergence to the Poisson kernel inside $\Omega_t$, unless $f$ is radial, may break down. Finally, let us point out that having proven the desired convergence in the $L^1$ norm for radial $\mathcal{C}_c(\mathbb{X})$ functions, one may conclude for the whole class of radial $L^1(\mathbb{X})$ initial data by a density argument, see \cite{APZ2023}.

		To begin with, let $x\notin \Omega_t$. Let also $\xi>0$ be a constant such that the compact support of $f$ is contained in $B(o, \xi)$. Then  we have
		\begin{align*}
			\int_{\mathbb{X}\,\smallsetminus\,\Omega_{t}}\,|e^{-t\sqrt{-\Delta}}f(x)|\,\diff{\mu}(x)\leq 
			\int_{B(o,\xi)}\,|f(y)|
			\int_{\mathbb{X}\,\smallsetminus\,\Omega_{t}}\,p_{t}(d(x,y))\,\diff{\mu}(x)\,\diff{\mu}(y).
		\end{align*}
		Notice that $x\in \mathbb{X}\smallsetminus \Omega _{t}$ and $y\in {%
			B(o,\xi)}$ imply $x\in \mathbb{X}\smallsetminus \widetilde{\Omega }_{t,y}$, where 
		\begin{equation*}
			\widetilde{\Omega }_{t,y}\,=\,\left\{ {x\in \mathbb{X}\,\big|\,2\,t^{2-\epsilon}\,\leq \,d(x,y)\,\leq \,\frac{1}{2}\,t^{2+\epsilon}}%
			\right\}
		\end{equation*}%
		provided $t$ is large enough. Indeed, for $t$ large enough, when $d(x,o)\geq t^{2+\epsilon}$ we have $d(x,y)\geq d(x,o)-d(y,o)>t^{2+\epsilon}-\xi>\frac{1}{2}\,t^{2+\epsilon}$, while when $d(x,o)\leq t^{2-\epsilon}$ we have $d(x,y)\leq d(x,o)+d(y,o)<t^{2-\epsilon}+\xi<2\,t^{2-\epsilon}$. Therefore, we have
		\begin{align*}
			\int_{\mathbb{X}\,\smallsetminus\,\Omega_{t}}\,|e^{-t\sqrt{-\Delta}}f(x)|\,\diff{\mu}(x)&\leq 
			\int_{B(o,\xi)}\,|f(y)|
			\int_{\mathbb{X}\,\smallsetminus\,\widetilde{\Omega }_{t,y}}\,p_{t}(d(x,y))\,\diff{\mu}(x)\,\diff{\mu}(y) \\
			&\lesssim t^{-\frac{\epsilon}{2}}\,\|f\|_{L^1(\mathbb{X})}\lesssim t^{-\frac{\epsilon}{2}},
		\end{align*} 
		working as in Proposition \ref{Poisson critical} for $\int_{\mathbb{X}\,\smallsetminus\,\widetilde{\Omega }_{t,y}}\,p_{t}(d(x,y))\,\diff{\mu}(x)$. Thus,
		\begin{align*}	\int_{\mathbb{X}\,\smallsetminus\,\Omega_{t}}\,|e^{-t\sqrt{-\Delta}}f(x)-M\, p_t(x)|\,\diff{\mu}(x)&\leq 	\int_{\mathbb{X}\,\smallsetminus\,\Omega_{t}}\,|e^{-t\sqrt{-\Delta}}f(x)|\,\diff{\mu}(x)+M\,	\int_{\mathbb{X}\,\smallsetminus\,\Omega_{t}}\,p_t(x)\,\diff{\mu}(x)\\
			&\lesssim t^{-\frac{\epsilon}{2}}.
		\end{align*}
		This proves the desired convergence outside the critical region for all $f\in \mathcal{C}_c(\mathbb{X})$.
		
		We now turn to $x\in \Omega_t$. By Lemma \ref{quotient}, the right-$K$-invariance of $\tau(k^{-1}.)$ and $f$, and the definition \eqref{H-Ftr} of the Helgason-Fourier transform that
		\begin{align}
			e^{-t\sqrt{-\Delta}}f(x)-M\,p_t(x)&= \int_{\mathbb{X}}\,(p_t(x,y)-p_t(x,o))\,f(y)\,\diff{\mu(y)} \notag \\
			&=p_t(x,o)\int_{\mathbb{X}}\,\left(\frac{p_t(x,y)}{p_t(x,o)}-1\right)\,f(y)\,\diff{\mu(y)} \notag \\
			&=p_t(x,o)\left\{\int_{\mathbb{G}}\left(e^{2\rho\, \tau(k^{-1}y_0)}-1+\textrm{O}(t^{-2+4\epsilon})\right)\,f(y_0)\,\diff{y_0}\right\} \notag \\
			&=
			p_t(x,o)\,
			\big(
			\widehat{f}(i\rho,k\mathbb{M})\,
			-\,\widehat{f}(-i\rho,k\mathbb{M})\,
			+\,\textrm{O}\big(t^{-2+4\epsilon}\,\|f\|_{L^1(\mathbb{X})}\big)
			\big).
			\label{S3 new rmk}
		\end{align}
		
		Notice that $\widehat{f}(\pm\,i\rho,k\mathbb{M})=\mathcal{H}f(\pm\,i\rho)=M$ when $f$ is radial, see Subsection \ref{subsection SS}. Therefore in this case, we deduce the desired convergence by integrating \eqref{S3 new rmk} over the critical region. 
		
		On the other hand, using the Cartan decomposition \eqref{vol} we have
		\begin{align*}
			\int_{\Omega_t}\,
			|e^{-t\sqrt{-\Delta}}f(x)\,-\,M\,p_{t}(x)|\, \diff{\mu(x)}\,
			&\longrightarrow\,
			\int_{K}
			\Big|
			\int_{\mathbb{G}}f(y_0)\,
			\big(
			e^{2\rho\, \tau(k^{-1}y_0)}\, \,-\,1
			\big)\, \diff{y_0}
			\Big|\,\diff{k}
		\end{align*}
		as $t\rightarrow+\infty$. The last integral is not constantly zero when $f$ is not radial. For example, consider $f$ to be a Dirac measure supported on some point $y=y_0K$ other than the origin, thus for $y_0\notin K$. In other words,  the solution now coincides with $p_t(\,.\,,y)$ and the mass is equal to $1$. In this case, however, the last integral is equal to $\int_{K}\Big|	e^{2\rho \, \tau(k^{-1}y_0)}\, \,-\,1
		\,\Big|\, \diff{k}$, thus does not vanish identically.
	\end{proof}	
	
	\textbf{Acknowledgments.} The author would like to thank the referee for useful comments which improved the presentation, and for pointing out the references \cite{CK03} and \cite{Shi17}. This work was supported by the Hellenic Foundation for Research and Innovation, Project HFRI-FM17-1733. Currently the author is funded by the Deutsche Forschungsgemeinschaft (DFG, German Research Foundation)--SFB-Gesch{\"a}ftszeichen --Projektnummer SFB-TRR 358/1 2023 --491392403.


\begin{thebibliography}{99}
		\bibitem{AbAl2022} Abadias L., Alvarez E.: Asymptotic behavior for the discrete in time heat equation.  \textit{Mathematics} \textbf{10} 3128 (2022).
		
		\bibitem{AGMP2021} Abadias L., González-Camus J., Miana P., Pozo J.:
		Large time behaviour for the heat equation on $\mathbb{Z}$, moments and decay rates. \textit{J. Math. Anal. Appl.} \textbf{500} (2021).
		
		
		\bibitem{AJ99} Anker J.-Ph., Ji L.: Heat kernel and Green function estimates on noncompact symmetric spaces. \textit{Geom. Funct. Anal.} \textbf{9}  1035--1091 (1999).
		
		\bibitem{AO} Anker J.-Ph., Ostellari P.: \textit{The heat kernel on noncompact symmetric spaces}. Lie Groups and Symmetric Spaces: In Memory of F.I. Karpelevich, Amer. Math. Soc. (2) Vol. 210 (2003).
		
		
		\bibitem{APZ2023} Anker J.-Ph.,  Papageorgiou E.,  Zhang H.-W.: Asymptotic behavior of solutions to the heat equation on noncompact symmetric spaces. \textit{J. Funct. Anal.} \textbf{284} (6) (2023).
		
		
		\bibitem{Bar1998}
		Barlow, M.: 
		\textit{Diffusions on fractals.} In: Lectures On Probability Theory And Statistics (Saint-Flour, 1995), pp. 1--121.
		Springer, Berlin (1998).
		
		
		\bibitem{BanEtAl} Banica V., Gonz{\'a}lez M.M.,  S{\'a}ez M.: Some constructions for the fractional Laplacian on noncompact manifolds. \textit{Rev. Mat. Iberoam.} \textbf{31} (2),  681--712 (2007).
		
		\bibitem{BSV17}  Bonforte, M.,  Sire Y.,  V{\'a}zquez J. L. Optimal existence and uniqueness theory for the fractional heat equation. \textit{Nonlinear Anal.} \textbf{153}, 142--168 (2017).
		
		
		\bibitem{BP2022} Bhowmik M., Pusti S.:
		An extension problem and Hardy's inequality for the fractional Laplace-Beltrami operator on Riemannian symmetric spaces of noncompact type.
		\textit{J. Funct. Anal.} \textbf{282} (9) (2022).
		
		
		\bibitem{BG1960}  Blumenthal R.M., Getoor R.K.: Some Theorems on Stable Processes, \textit{Trans. Amer. Math. Soc.} \textbf{95}, 263--273 (1960).
		
		\bibitem{BJ07} Bogdan K., Jakubowski T.: Estimates of the heat kernel of fractional Laplacian perturbed by gradient operators, \textit{Commun. Math. Phys.} \textbf{271} (1), 179--198 (2007).
		
		
		\bibitem{CaSi2007} Caffarelli L., Silvestre L.: An extension problem related to the fractional Laplacian. \textit{Commun. Partial Differ. Equ.} \textbf{32} (7–9)  1245--1260 (2007).
		
		
		\bibitem{ChHa2020}
		Chen X., Hassell, A.:
		The heat kernel on asymptotically hyperbolic manifolds. 
		\textit{Commun. Partial Differ. Equ.} \textbf{45}, 1031--1071 (2020).
		
		\bibitem{CK03} Chen Z.-Q.,  Kumagai T.: Heat kernel estimates for stable-like processes on $d$-sets,
		\textit{Stoch. Proc. Their Appl.}, \textbf{108}, no. 1, 27--62 (2003).
		
		
		\bibitem{CK08}   Chen Z.-Q.,  Kumagai T.: Heat kernel estimates for jump processes of mixed types on metric measure spaces, \textit{Probab. Theory Related Fields} \textbf{140}, 277--317 (2008).
		
		
		\bibitem{CowGM} Cowling M.G., Giulini S., Meda S.: $L^p-L^q$ estimates for functions
		of the Laplace-Beltrami operator on noncompact symmetric spaces II,
		\textit{J. Lie Theory} \textbf{5}, 1--14 (1995).
		
		
		\bibitem{DaMa1988}
		Davies, E., Mandouvalos, N.:
		Heat kernel bounds on hyperbolic space and Kleinian groups. 
		\textit{Proc. London Math. Soc.} (3) \textbf{57}, 182--208 (1988).
		
		
		\bibitem{DzPr2018}
		Dziuba{\'n}ski, J., Preisner, M.:
		Hardy spaces for semigroups with Gaussian bounds. 
		\textit{Ann. Mat. Pura Appl.} (4). \textbf{197}, 
		965--987 (2018).
		
		\bibitem{FaSt1986} Fabes, E., Stroock, D.:
		A new proof of Moser's parabolic Harnack inequality using the old ideas of Nash. 
		\textit{Arch. Ration. Mech. Anal.} \textbf{96}, 327--338 (1986).
		
		
		\bibitem{G61}  Getoor R.K.: Infinitely divisible probabilities on the hyperbolic plane, \textit{Pacific J. Math.} \textbf{11}  1287–1308 (1961).
		
		\bibitem{GS04} Graczyk P., Stos A.: Transition density estimates for stable processes on symmetric spaces, \textit{Pacific J. Math.} \textbf{217}  87-100 (2004).
		
		\bibitem{Gri1991}
		Grigor'yan, A.:
		The heat equation on noncompact Riemannian manifolds. 
		\textit{Mat. Sb.} \textbf{182}, 55-87 (1991).
		
		
		\bibitem{Gri2009} Grigor'yan, A.:
		\textit{Heat kernel and analysis on manifolds. }
		AMS International Press (2009).
		
		
		\bibitem{Gri1995} Grigor’yan, A.: Upper bounds of derivatives of the heat kernel on an arbitrary complete
		manifold. \textit{J. Funct. Anal.} \textbf{127}, 363-389 (1995).
		
		
		\bibitem{GPZ2022} Grigor’yan A., Papageorgiou E.,  Zhang H.-W.: Asymptotic behavior of the heat semigroup on certain Riemannian manifolds. To appear in "\textit{From Classical Analysis to Analysis on Fractals - The Robert Strichartz Memorial Volume}", Springer, 2023.	ArXiv: 2205.06105.
		
		\bibitem{Hel}  Helgason S.: \textit{Geometric Analysis on Symmetric Spaces}, Mathematical Surveys and Monographs, Second Edition. Amer. Math. Soc., 637 pp (2008).
		
		\bibitem{LiYa1986}
		Li, P., Yau, S.:
		On the parabolic kernel of the Schr{\"o}dinger operator. \textit{Acta Math.} \textbf{156}, 153-201 (1986).
		
		\bibitem{Mesi2011}  Meerschaert M., Sikorskii A.: \textit{Stochastic Models for Fractional Calculus}. De Gruyter (2011).
		
		
		\bibitem{P2023} Papageorgiou E.: Asymptotics for the infinite Brownian loop on noncompact symmetric spaces. Preprint (2023). Arxiv: 2301.09924.
		
		
		
		\bibitem{Sal1992}
		Saloff-Coste, L.:
		A note on Poincar{\'e}, Sobolev, and Harnack inequalities.
		\textit{Internat. Math. Res. Notices}, 27--38 (1992).
		
		\bibitem{Sal2002}
		Saloff-Coste, L.:
		\textit{Aspects of Sobolev-type inequalities.}
		Cambridge University Press (2002).
		
		\bibitem{Sal2010}
		Saloff-Coste, L.:
		\textit{The heat kernel and its estimates.}
		In: Probabilistic Approach To Geometry, pp. 405--436.
		Adv. Stud. Pure Math. \textbf{57}, Math. Soc. Japan, Tokyo (2010).
		
		\bibitem{Shi17} Shiozawa, Y.: Bottom crossing probability for symmetric jump processes, \textit{Math. Z.}
		\textbf{287}, 1355--1376 (2017).
		
		\bibitem{Stinga} Stinga, P.R.: ``User’s guide to the fractional Laplacian and the method of semigroups''. In: Volume 2 Fractional Differential Equations, edited by Anatoly Kochubei and Yuri Luchko, Berlin, Boston: De Gruyter, pp. 235-266 (2019).
		
		\bibitem{ST2010}  Stinga P.R.,  Torrea J.L.: Extension problem and Harnack’s inequality for some fractional operators. \textit{Commun. Partial Differ.}  \textbf{35} (10-12),  2092--2122 (2012).
		
		
		\bibitem{Str1983}
		Strichartz, R.:
		Analysis of the Laplacian on the complete Riemannian manifold.
		\textit{J. Funct. Anal.} \textbf{52}, 48-79 (1983).
		
		
		\bibitem{Vaz2017}  V{\'a}zquez J.L.: Asymptotic behaviour methods for the heat equation. Convergence to the Gaussian. Course Notes (2017). ArXiv: 1706.10034.
		
		\bibitem{Vaz2018}  V{\'a}zquez J.L.:	Asymptotic behaviour for the fractional Heat equation in the Euclidean space. \textit{Complex Var. Elliptic Equ.}, Special volume in honor of Vladimir I. Smirnov's 130th anniversary, \textbf{63 (7-8)} (2018), 1216--1231.
		
		\bibitem{Vaz2019} V{\'a}zquez J.L.: Asymptotic behaviour for the heat equation in hyperbolic space.  \textit{Comm. Anal. Geom.} \textbf{30} (9) 2123--2156 (2022).
		
		
		\bibitem{Y}  Yosida K.: \textit{Functional Analysis}. Springer, Berlin (1980).
	\end{thebibliography}
\end{document}